\documentclass[12pt,reqno]{amsart}

\usepackage[arrow,matrix,curve]{xy}

\usepackage[dvips]{graphicx} 

\usepackage{amssymb, latexsym, amsmath, amscd, array, hyperref, epigraph
}

\usepackage{framed,multicol} 

\newtheorem{theorem}{Theorem}[section]

\theoremstyle{definition}
\newtheorem{definition}[theorem]{Definition}
\newtheorem{example}[theorem]{Example}
\newtheorem{remark}[theorem]{Remark}

\numberwithin{equation}{section}

\newcommand\N {{\mathbb N}} 

\newcommand\R {{\mathbb R}}
\newcommand\Q {{\mathbb Q}}
 
\newcommand\st{{\rm st}} 

\newcommand\astr{{{}^{\ast}\hspace{-0.5pt}\R}}

\newcommand{\hr} {{{}^{\mathfrak{h}}\hspace*{-0.4pt}\R}}

\DeclareMathOperator{\adequal}{\;\raisebox{-3pt}{$\ulcorner\!\urcorner$}\;}

\newcommand\parisotes{{$\pi\alpha\rho\iota\sigma\acute{o}\tau\eta\varsigma$}}

\newcommand\parisos{{$\pi\acute\alpha\rho\iota\sigma{o}\varsigma$}}

\numberwithin{equation}{section}

\author[J. Bair]{Jacques Bair}\address{J. Bair, HEC-ULG, University of
Liege, 4000 Belgium}\email{J.Bair@ULiege.be}

\author[M. Katz]{Mikhail G. Katz}\address{M. Katz, Department of
Mathematics, Bar Ilan University, Ramat Gan 5290002 Israel}
\email{katzmik@macs.biu.ac.il}

\author[D. Sherry]{David Sherry}\address{D. Sherry, Department of
Philosophy, Northern Arizona University, Flagstaff, AZ 86011, US}
\email{David.Sherry@nau.edu}

\begin{document}

\thispagestyle{empty}


\title[Fermat's dilemma: Why did he keep mum on infinitesimals?]
{Fermat's dilemma: Why did he keep mum on infinitesimals? and the
European theological context}

\begin{abstract}
The first half of the 17th century was a time of intellectual ferment
when wars of natural philosophy were echoes of religious wars, as we
illustrate by a case study of an apparently innocuous mathematical
technique called \emph{adequality} pioneered by the honorable
\emph{judge} Pierre de Fermat, its relation to indivisibles, as well
as to other hocus-pocus.  Andr\'e Weil noted that simple applications
of adequality involving polynomials can be treated purely
algebraically but more general problems like the cycloid curve cannot
be so treated and involve additional tools--leading the
\emph{mathematician} Fermat potentially into troubled waters.  Breger
attacks Tannery for tampering with Fermat's manuscript but it is
Breger who tampers with Fermat's procedure by moving all terms to the
left-hand side so as to accord better with Breger's own interpretation
emphasizing the \emph{double root} idea.  We provide modern proxies
for Fermat's procedures in terms of relations of infinite proximity as
well as the standard part function.

\end{abstract}

\maketitle

\tableofcontents

\epigraph{Father Bertet%
\footnote{See further in note~\ref{f13}.}
sayth your Bookes are in great esteeme, but not to be procured in
Italy.  (Collins to Gregory)}

\section{Introduction}

In his review of Pietro Redondi's book on Galileo, Egidio Festa
writes: ``Avec la sentence du Tribunal du Saint-Office, la
fracture--qui allait se r\'ev\'eler irr\'eversible--entre v\'erit\'e
de science et v\'erit\'e de foi apparut au grand jour''%
\footnote{``Following the verdict of the Tribunal of the Holy Office,
the split--which would prove to be irreversible--between the truth of
science and the truth of faith appeared in broad daylight.''}
\cite[p.\;94]{Fe91}.  Whatever the merits of Festa's generalisation
concerning a conflict opposing science and catholicism (surely the
pagan Aristotelian hylomorphism had to give way to make room for the
emerging atomist science), many modern studies have focused on such
tensions in the context of Galileo and other scientists (including the
jesuit condemnation of indivisibles in the works of Galileo,
Cavalieri, and others).  The scientific community in France closely
followed the 1633 trial; Pierre Gassendi's role is analyzed in
\cite{Fe91}; on Carcavy and Mersenne see Section~\ref{s411}.

Mathematicians are sometimes thought to have been affected to a lesser
extent, though well-known examples include Cavalieri and degli Angeli.
In this text we examine evidence that doctrinal factors may have
affected the presentation of the work of a mathematician not usually
thought of as involved in the Galilean controversy: the honorable
judge Pierre de Fermat.

Fermat's jesuit friend Antoine de Lalouv\`ere at Toulouse was a
``declared enemy'' of indivisibles \cite[p.\;267]{De15}.  Had
Lalouv\`ere expressed similar sentiments about them to Fermat as he
did in his mathematical works, it would have provided little
encouragement for Fermat to be overly explicit concerning the
foundation of his method of \emph{adequality}.

Fermat's professional activities at the Parliament of Toulouse
included occasional stints at the \emph{Chambre de l'\'Edit} at a
nearby town of Castres.  The relevance of the \emph{Chambre}
specifically is the fact that it was the organ of the Parliament that
dealt with interdenominational quarrels (mainly of a pecuniary type),
with judges split evenly between catholics and protestants.  A key
disagreement between catholics and protestants concerned the
interpretation of the eucharist and its relation to atomism and
kindred scientific doctrines like indivisibles and infinitesimals.

Due to the sensitive nature of his engagement at Castres, Fermat may
have found it risky to speak freely of the nature of quantities
involved in his technique of \emph{adequality}.  For his part, jesuit
Tacquet claimed that one must destroy the theory of indivisibles (if
Euclidean geometry is to be saved).  If so, legitimate questions could
arise concerning its practitioners' soundness of judgment and their
fitness to judge cases that come before all-powerful Parliaments.

\subsection{A re-evaluation}
\label{s11}

A re-evaluation of Fermat's technique of adequality and its place in
the history of mathematics involves at least three components, which
we describe respectively as the weak, standard, and strong theses.
\begin{enumerate}
\item
\label{t1}
(weak thesis) the fact itself that Fermat refrained from elaborating
in detail on the foundations of his technique, and specifically
refrained from speculating on the nature of his~$E$ used in the
technique, does not constitute evidence toward a hypothesis that the
technique was a purely algebraic one, due to professional and
religious constraints Fermat was operating under, that would have made
any such speculations reckless as far as his professional career was
concerned.
\item
\label{t2}
(standard thesis) We challenge the view adopted by some scholars
working in a default Weierstrassian interpretive scheme, that Fermat
used a purely algebraic technique unrelated to any notion of
approximation or to infinitesimals.  We argue that internal evidence
in Fermat's work on the cycloid, on integration, and elsewhere points
to a contrary conclusion (see Weil's thesis in Section~\ref{s13b}).
While Fermat may speak of his~$E$ cautiously as being an
\emph{arbitrary} quantity, in his actual mathematical practice he
exploits it as an \emph{arbitrarily small} quantity that can be
discarded in certain calculations.
\item
\label{t3}
(strong thesis) We argue that the procedures of Fermat's work on
maxima and minima as well as tangents, centers of gravity, and
refraction find better proxies in the procedures of Robinson's
infinitesimal-enriched framework \cite{Ro66} than in traditional
Weierstrassian frameworks stripped of infinitesimals, just like the
procedures of Gregory, Leibniz, Euler, Cauchy, and others.
\end{enumerate}

The present article deals mainly with arguing thesis~\eqref{t1}.
Evidence toward thesis~\eqref{t2} appears in Sections~\ref{s2} and
following.  Evidence toward thesis~\eqref{t3} appears in
Section~\ref{s110}.%
\footnote{Cf. \cite{13e}, \cite[Section\;3]{14a},
\cite[Section\;4]{17d}.}
The thesis that Fermat's procedures are more akin to modern
infinitesimal theories than to classical Weierstrassian analysis is
developed in \cite{Ci90} in the context of the Lawvere--Kock--Reyes
framework; see \cite{Ko06}.

\subsection{Procedures versus ontology}
\label{f2}

Cifoletti seeks to justify ``[l]e fait de juxtaposer deux fragments de
th\'eories math\'ematiques s\'epar\'es par plus de trois cent
cinquante ans'' \cite[p.\;4]{Ci90}.  We would like to add the
following two points to her arguments.

(A) Traditionally trained historians of mathematics apply conceptual
frameworks deriving from modern Weierstrassian foundations, separated
from Fermat by 250 years, in order to interpret Fermat; thus, the
interpretation in \cite[p.\;17]{Ba11} is explicitly stated to be based
on an implicit function theorem of Ulisse Dini (1845--1918).

(B) The procedure versus ontology dichotomy enables us to seek proxies
for Fermat's techniques while acknowledging the difference between
ontological foundations of pre-Weierstrassian mathematics and modern
mathematics, whether set-theoretic (as in the case of Weierstrass and
Robinson) or category-theoretic (as in the case of Lawvere).%
\footnote{These two themes were explored more fully in recent articles
in \emph{Erkenntnis} \cite{13f}, \emph{HOPOS} \cite{16a},
\emph{Journal for General Philosophy of Science} \cite{17a}, and
elsewhere.}
On modern proxies see further in Section~\ref{s19}.

\subsection{Adequality and the cycloid curve}
\label{s2}

Fermat treated numerous problems concerning maxima and minima as well
as tangents to curves using a procedure (more precisely, a cluster of
procedures) called \emph{adequality}.  Fermat's work became known
mainly through private letters to correspondents like Carcavy and
Mersenne, and was eventually popularized through inclusion in a book
by Pierre H\'erigone \cite[p.\;2]{Ci90}.

A typical application involves finding the maximum of an expression
involving~$x$ (Fermat used ``$A$'') that we will denote in modern
notation by~$p(x)$.  Fermat compared~$p(x)$ and~$p(x+E)$ and after
doing algebra to eliminate common terms, he would form a relation
(\emph{adequality}) among terms divisible by~$E$.  He would then
cancel out a factor of~$E$ and then discard any remaining terms in~$E$
to obtain the answer%
\footnote{To a modern reader familiar with the calculus the procedure
is reminiscent of quotients of increments occurring in the definition
of the derivative but the period under discussion precedes the
calculus of Newton and Leibniz.}
(a more detailed example appears in Section~\ref{s110}).  Since the
relation between~$p(x)$ and~$p(x+E)$ is not one of true equality,
Fermat referred to it as \emph{adequality}.

The term \emph{adequality} originates with Diophantus' \parisotes{}
(meaning \emph{approximate equality}) and \parisos{} rendered as
\emph{adaequo} in Bachet's translation of Diophantus'
\emph{Arithmetic}, with which Fermat was intimately familiar.  In the
text sent to Descartes in the late 1630s Fermat explicitly names
Diophantus as the source of the term: ``Adaequentur, ut loquitur
Diophantus''%
\footnote{\label{f5}``[Let there be] adequated, to speak like
Diophantus'' (here we use \emph{adequate} as a verb--rather than as an
adjective--in a neologism aiming to convey the meaning of the Latin).
See further in Section~\ref{f14}.}
\cite[p.\;133]{Ta1}.
 
For a transcendental curve like the cycloid (see e.g.,
\cite[p.\;70]{Ci90}), Fermat solved the problem of finding the tangent
line at an arbitrary point of the curve as follows (see
\cite[p.\;144]{Ta3}).  Start with the \emph{defining equation} of the
cycloid (Fermat refers to such an equation as \emph{la propri\'et\'e
sp\'ecifique de la courbe}).  Consider the tangent line (to be
determined) at a point~R of the curve.  Choose a nearby point~N
\emph{on the tangent line}.  Denote by~D and~E respectively the
projections of~R and~N onto a suitable axis.  Denote by~$E$ the
distance between the points D and E.  Substitute the point~N into the
defining equation of the cycloid \emph{as if} it satisfied the latter
(hence not \emph{equality} but \emph{adequality}), and proceed with
the steps involving cancellations of the quantity~$E$; see
\cite[Section\;5]{13e} for details.%
\footnote{Herbert Breger attempts to account for Fermat's treatment of
the cycloid without mentioning any notion of smallness but ends up
falling back on the condition ``if the point E is not too far away of
the point D, then\ldots'' \cite[p.\;206]{Br94}.  Breger's biased
criticism of Paul Tannery's Fermat scholarship is dealt with in
Section~\ref{f9b}.  Breger's flawed rendering of Fermat's method is
dealt with in Section~\ref{s111}.}

Since the cycloid is a transcendental curve, it would be difficult to
interpret Fermat's solution as a purely algebraic procedure involving
a formal symbol~$E$,
%
%
as already noted by Andr\'e Weil:
\begin{quote}
At first Fermat applies the method only to polynomials, in which case
it is of course purely algebraic; later he extends it to increasingly
general problems, including the cycloid \cite[p.\;1146]{We73}.
\end{quote}

\subsection{Weil's thesis}
\label{s13b}

We will elaborate Weil's remark cited above as the following
principle:
\begin{quote}
\emph{Weil's thesis}: simple applications of Fermat's method that
involve polynomials can be treated purely algebraically, but more
general problems (like the cycloid) cannot be so treated and involve
tools that of a more geometric or analytic flavor.%
\footnote{Though apparently straighforward, Weil's thesis has been
resisted by a small number of Fermat scholars recently; see
Section~\ref{s13}.}
\end{quote}
Approximate equality similarly plays a role in Andersen's
interpretation of Fermat's application of adequality to derive Snell's
law of refraction.  The procedure involves discarding
\emph{second-order terms} in~$E$ twice \cite[p.\;55]{An83}.  Such a
procedure is only meaningful for small~$E$ when higher-order terms are
negligible compared to~$E$ itself.  Thus Fermat's treatment of the
cycloid and Snell's law furnishes \emph{textual} evidence that his use
of~$E$ relies on it being \emph{small}.%
\footnote{We therefore reject Breger's claim that ``[t]he idea that
the mathematically conservative Fermat \ldots{} is supposed to have
calculated using infinitesimals is a bold hypothesis for which there
are no textual proofs'' \cite[p.\;23]{Br13}.}

Mahoney assumes matter-of-factly that Fermat used infinitesimals%
\footnote{Without referring to them explicitly as \emph{infinitely
small}, of course.  We should mention that Mahoney's interpretation of
Fermat's method of adequality is incompatible with ours.}
at least in his method of quadrature:
\begin{quote}
Especially where, in his mature method of quadrature, Fermat began to
operate with infinitesimals and limit procedures, the reversion to
synthesis became, as mathematicians from Newton to Cauchy were to
find, far from easy.  \cite[p.\;47]{Ma94} 
\end{quote}
Scholars ranging from Sabetai Unguru to Andr\'e Weil%
\footnote{The pair may be familiar to the reader as the leaders of
opposing schools in the debate over \emph{geometric algebra}; see
e.g., \cite{We78}.}
take it for granted that Fermat's technique involves approximate
equality.  Thus, Unguru writes in reference to Fermat's method of
tangents:
\begin{quote}
This computation involved\ldots{} the concept of `adequality' (taken
over from Diophantus), according to which\ldots{} the two ordinates of
(1)\;an arbitrary point on the tangent other than the point of
tangency and (2) the corresponding point on the curve determined by
the intersection of that ordinate and the curve were assumed to be
`adequal' \cite[p.\;776]{Un76}.
\end{quote}
Since the point on the tangent line and the corresponding point on the
curve (with the same ordinate) are in general \emph{unequal}, we see
that according to Unguru's reading, setting them \emph{adequal}
involves an approximate equality.  Unguru comments further:
\begin{quote}
Initially inspired by Archimedes' treatment in `On spirals', Fermat
later improved and generalized his method of quadrature by means of a
new meaning attached to the concept of `adequality', now taken to
stand for `approximate' or `limiting' equality. (ibid.)
\end{quote}
Similarly, Weil notes that Diophantus uses the Greek term
\begin{quote}
to designate his way of \emph{approximating} a given number by a
rational solution to a given problem (cf.~e.g.~\emph{Dioph.}\;V.11 and
14) \cite[p.\;28]{We84}. (emphasis added)
\end{quote}
According to Weil's interpretation, \emph{approximation} is inherent
to the meaning of the original Greek term \parisotes.

Breger appears to recognize that Leibniz understood Fermat's method to
be based on infinitesimals, but claims that ``the texts available to
Leibniz contained grave contradictions and were hardly suitable for a
clear understanding of the method'' \cite[p.\;24]{Br13}.  Yet Breger
acknowledges that Leibniz may have had access to the 1679 \emph{Varia
Opera} edition of Fermat's collected works (ibid., note 28) and
moreover that Huygens and Leibniz corresponded about the \emph{Opera
Omnia} edition (ibid., note\;31).

\subsection{Reception by Huygens}

The thesis we develop in the present text does not depend on
considering Fermat's~$E$ as necessarily being infinitely small.  Even
if one adopts a purely algebraic interpretation of Fermat's techniques
including quadrature and adequality (as is eminently possible for some
of Fermat's simpler applications of adequality, as per Weil's thesis;
see Section\;\ref{s13b}), one can still ask \emph{why} Fermat, unlike
his contemporaries like Kepler%
\footnote{\label{f8b}Jean d'Espagnet was interested in the work of
Robert Fludd (1574--1637), whom Kepler had attacked.  That alone might
have brought d'Espagnet to Kepler's books.  While it is plausible that
d'Espagnet's library at Bordeaux should have included Kepler's works,
the question of Kepler's possible influence on Fermat's method is a
subject of long-standing scholarly controversy exhaustively covered in
\cite[pp.\;40--60]{Ci90}.  For a summary of Kepler's contribution to
infinitesimal techniques of barrel measuring see \cite{Jo08}.}
and Galileo, chose \emph{not} to work explicitly with the infinitely
small.

Fermat offered little in the way of explanation of his method of
adequality.  What are the reasons for Fermat's evasiveness about the
foundations of his method?  Why did Fermat never clarify whether
his~$E$ were infinitely small or account for some of his methods in
terms of indivisibles, as Kepler, Galileo, and Cavalieri did?  While
presenting an interpretation of Fermat's method at the French academy,
Huygens commented as follows:
\begin{quote}
Fermat est le premier homme que je sache qui ait \'etabli une r\`egle
certaine pour d\'eterminer les valeurs maximales et minimales dans les
questions g\'eom\'etriques.  En en recherchant \emph{le fondement
qu'il n'a pas communiqu\'e}, etc.  (Huygens cited in \cite{NT})
(emphasis added)
\end{quote}
Most commentators agree with Huygens' sentiment.  Huygens continued:
\begin{quote}
\ldots j'ai trouv\'e en m\^eme temps de quelle mani\`ere cette r\`egle
peut \^etre r\'eduite \`a une bri\`evet\'e remarquable, de sorte
qu'elle s'accorde d\'esormais avec celle donn\'ee plus tard par
l'honorable Hudde comme une partie de la r\`egle plus g\'en\'erale et
fort \'el\'egante qui s'appuie sur un tout autre principe.  Cette
derni\`ere a \'et\'e publi\'ee par Fr.\;Van Schooten dans le recueil
qui contient aussi les livres de Descartes sur la g\'eom\'etrie.  Or,
ma m\'ethode d'examiner la r\`egle de Fermat \'etait la suivante\ldots
(Oeuvres compl\`etes, 1940, tome 20).  (ibid.)
\end{quote}
Huygens' observation that Fermat did not provide an explanation of the
foundation of his method is echoed by De Gandt:
\begin{quote}
Fermat et Roberval, pour des raisons tr\`es diff\'erentes, conservent
une grande partie de leurs tr\'esors dans leurs papiers personnels.
\cite[p.\;104]{De92}
\end{quote}
Concerning Fermat's limited mathematical output, Spiesser writes:
\begin{quote}
[Fermat] s'en explique par le manque de temps, \ldots et par `sa pente
naturelle vers la paresse' (par exemple, Oeuvres II, p.\;461).  Ces
arguments, dont le second rel\`eve quelque peu de la coquetterie, lui
permettent \'egalement d'excuser son manque d'int\'er\^et pour les
expositions d\'etail\-l\'ees des preuves \ldots{} \cite[p.\;297]{Sp16}
\end{quote}
Spiesser's assessment of Fermat's claim of alleged \emph{paresse} as a
kind of \emph{coquetterie} is right on target.  We will therefore seek
deeper reasons for his apparent evasiveness in explaining his method.

\subsection{Our thesis}

We shall argue that Fermat's reticence may have been due either to a
deference to his contemporary catholic theologians including jesuits
who were often opposed to atomism and indivisibles on doctrinal
grounds, or to a possible fear of religious reprisals.  If this is
correct, then Fermat was in similar position to 17th century Italian
mathematicians most of whom forsook indivisibles.%
\footnote{A possible exception is the work of Mengoli,
cf.\;\cite{Ma97}, though Mengoli carefully avoided the language of
indivisibles.}

\subsection{Modern proxies}
\label{s19}

In this and the following sections we outline a modern formalisation
of a relation of adequality and the associated procedures while
keeping in mind the distinction between procedure and ontology
outlined in Section~\ref{f2}.  Readers already familiar with this
set-theoretic justification of an infinitesimal-enriched continuum can
skip to Section~\ref{ias}.

We start with a construction (called an \emph{ultrapower}) of a
hyperreal extension~$\R\hookrightarrow\astr$.  Let~$\R^{\N}$ denote
the ring of sequences of real numbers, with arithmetic operations
defined termwise.  Then we have a totally ordered
field~$\astr=\R^{\N}\!/\text{MAX}$ where ``MAX'' is a suitable maximal
ideal.  Elements of~$\astr$ are called hyperreal numbers.  Note the
formal analogy between the quotient~$\astr=\R^{\N}\!/\text{MAX}$ and
the construction of the real numbers as equivalence classes of Cauchy
sequences of rational numbers.  In both cases, the subfield is
embedded in the superfield by means of constant sequences, and the
ring of sequences is factored by a \emph{maximal ideal}.

We now describe a construction of such a maximal ideal of $\R^\N$
exploiting a suitable finitely additive
measure~$\xi\colon\mathcal{P}(\N)\to\{0,1\}~$ (thus~$\xi$ takes only
two values,~$0$ and~$1$) taking the value~$1$ on each cofinite set,%
\footnote{For each pair of complementary \emph{infinite} subsets
of~$\N$, such a measure~$\xi$ ``decides'' in a coherent way which one is
``negligible'' (i.e., of measure~$0$) and which is ``dominant''
(measure~$1$).}
where~$\mathcal{P}(\N)$ is the set of subsets of~$\N$.  The ideal MAX
consists of all ``negligible'' sequences~$(u_n)$, i.e., sequences
which vanish for a set of indices of full measure~$\xi$, namely,
$\xi\big(\{n\in\N\colon u_n=0\}\big)=1$.  The
subset~$\mathcal{U}=\mathcal{U}_\xi\subseteq\mathcal{P}(\N)$
consisting of sets of full measure~$\xi$ is called a free ultrafilter
(these can be shown to exist using Zorn's lemma).

\begin{definition}
The \emph{order} on~$\astr$ is defined by setting~$[(u_n)]<[(v_n)]$ if
and only if~$\xi(\{n\in\N\colon u_n<v_n\})=1$ or
equivalently~$\{n\in\N\colon u_n<v_n\}\in\mathcal{U}$.
\end{definition}

\begin{example}   
An element~$x\in\R$ is embedded in~$\astr$ by means of the constant
sequence with general term~$x$.  Let~$v=[(v_n)]\in\astr$.  Then~$x$
satisfies~$x<v$ if and only if~$\{n\in\N\colon x<v_n\}\in\mathcal{U}$.
\end{example}

For a broader view of Robinson's framework see \cite{17f}.  Minimal
set-theoretic requirements for constructing a \emph{definable}
hyperreal line are analyzed in \cite{18c}.

\subsection{Infinitesimals, adequality, and standard part}
\label{ias}

Based on the set-theoretic construction of a hyperreal extension
$\R\hookrightarrow\astr$ given in Section~\ref{s19}, we introduce some
terminology for dealing with infinitesimals.

\begin{definition}
An element~$E\in\astr$ is called \emph{infinitesimal} if for each
positive~$r\in\R$ one has~$-r<E<r$.
\end{definition}

\begin{definition}
Hyperreal numbers~$a,b$ are said to be infinitely close, written
\[
a\approx b,
\]
if their difference~$a-b$ is infinitesimal.
\end{definition}

It is convenient also to introduce the following terminology and
notation.  We will use Leibniz's notation~$\adequal$.%
\footnote{Leibniz actually used a symbol that looks more like~$\sqcap$
(see \cite{13f}) but the latter is commonly used to denote a product.}
Leibniz used the symbol to denote a notion of \emph{generalized
equality} ``up to'' a negligible term, though he did not distinguish
it from the usual symbol~``$=$'' which he also used in the same sense.

\begin{definition}
\label{d356}
Hyperreal numbers~$a,b$ are said to be \emph{adequal}, written
\[
a \adequal b,
\]
if either~$\frac{a}{b}\approx 1$ or~$a=b=0$.
\end{definition}

\begin{example}
We have $\sin x\adequal x$ for infinitesimal $x$. While the
relation~$\sin x\approx x$ for infinitesimal~$x$ is immediate from the
continuity of sine at the origin (in fact both sides are infinitely
close to~$0$), the relation~$\sin x\adequal x$ is a subtler relation
equivalent to the computation of the first order Taylor approximation
of sine.
\end{example}

The crucial observation is that unlike the relation~$\approx$, the
relation~$\adequal$ has the following property of
\emph{multiplicative} invariance.

\begin{theorem}
\label{t15}
The relation~$\adequal$ is multiplicatively invariant in the sense
that if polynomials (or more general expressions)~$P$ and~$Q$
satisfy~$P\adequal Q$ then also~$\frac{P}{E}\adequal\frac{Q}{E}$.
\end{theorem}

\begin{proof}
If~$\frac{P}{Q}\approx 1$ then
$\frac{P/E}{Q/E}=\frac{EP}{EQ}=\frac{P}{Q}\approx 1$, as well.
\end{proof}

An element~$u\in\astr$ is called \emph{finite} if~$-r<u<r$ for a
suitable~$r\in\R$.  Let~$\hr\subseteq\astr$ be the subring consisting
of finite elements of~$\astr$.

\begin{theorem}
\label{t17}
There exists a function~$\st\colon\hr\to\R$ called \emph{the standard
part} that rounds off each finite hyperreal~$u$ to its nearest real
number~$u_0\in\R$, so that~$u_0=\st(u)$ and~$u\approx u_0$.
\end{theorem}

\begin{proof}
The result holds generally for an arbitrary ordered field
extension~$F$ of~$\R$.  Indeed, if~$x\in F$ is finite, then~$x$
induces a Dedekind cut on the subfield~$\Q\subseteq \R \subseteq F$
via the total order of~$F$.  The real number corresponding to the
Dedekind cut is then infinitely close to~$x$.
\end{proof}

\begin{example}
A number~$E$ is infinitesimal if and only if~$\st(E)=0$.
\end{example}

\subsection{Transcription of Fermat's argument}
\label{s110}

Fermat illustrates his method by solving the following problem:
subdivide a given segment into two subsegments so as to maximize the
area of the rectangle formed by the two subsegments
\cite[p.\;134]{Ta1}.

Fermat denotes the length of the original segment by~$B$ and the
length of a first subsegment by~$A$, so that the second subsegment is
of length~$B-A$.  The expression to be maximized is
\[
BA-A^2,
\]
where we used the Cartesian notation for arithmetic operations for
greater readability.  Fermat now replaces the first subsegment
by~$A+E$ so that the second one becomes~$B-A-E$, with
product~$BA-A^2+BE-2AE-E^2$.  Fermat compares the latter to the
previous product~$BA-A^2$ and suppresses common terms.

\begin{remark}
\label{r19}
At the next stage, Fermat gathers together terms with of same sign,
namely he groups the positive terms together and the negative terms
together.
\end{remark}

This enables him to form the \emph{adequality}
\begin{equation}
BE \text{ adaequabitur } 2AE+E^2,
\end{equation}
Dividing out by~$E$, Fermat obtains the \emph{adequality}
\begin{equation}
B \text{ adaequabitur } 2A+E.
\end{equation}
Suppressing the remaining term~$E$, Fermat obtains an \emph{equality}
\begin{equation}
B \text{ aequabitur } 2A.
\end{equation}
Fermat concludes that to solve the problem one needs to take (for~$A$)
half of~$B$.  
\begin{remark}
\label{r110}
This is one of many instances when the characteristic pair
``adequality, equality'' (in that order) occurs in Fermat's
description of his method.
\end{remark}

Let us transcribe Fermat's argument in the language developed in
Section~\ref{ias} by rewriting his sequence of operations in modern
notation as follows:
\begin{equation}
\label{e11}
BE\adequal 2AE+E^2,
\end{equation}
\begin{equation}
\label{e12}
B\adequal 2A+E,
\end{equation}
\begin{equation}
\label{e12b}
B=2A.
\end{equation}
In terms of modern proxies, the passage from relation~\eqref{e11} to
relation~\eqref{e12} is justified by Theorem~\ref{t15}.  The passage
from relation~\eqref{e12} to equality~\eqref{e12b} is justified by
applying the standard part function of Theorem~\ref{t17} using the
fact that~$\st(E)=0$.

\begin{remark}
\label{r16}
The passage from equation~\eqref{e11} to equation~\eqref{e12} enabled
by Theorem~\ref{t15} is only possible if one works with a
\emph{multiplicatively} invariant relation~$\adequal$ as we did,
rather than with the relation of infinite proximity~$\approx$ (see
Section~\ref{ias}).
\end{remark}

\subsection{Text sent to Descartes}
\label{s21}

At the outset of a text on Fermat's method sent to Descartes toward
the end of 1637 and received by the latter at the beginning of 1638,
Fermat starts with a term (expression) containing an unknown~$A$.%
\footnote{\label{f13c}In \emph{Ad eamdem methodum} (item III in
\cite{Ta1}), Fermat refers to~$A$ as \emph{incognitam} (unknown)
\cite[p.\;140]{Ta1}.  A text starting with the words ``Je veux par ma
m\'ethode'' is a translation (of the above) dating from 1638 in old
French first published in \cite{Fe22}.  This text refers to~$A$ as an
\emph{inconnue} \cite[p.\;74, line\;7]{Fe22} and similarly at
\cite[p.\;80, line\;4]{Fe22}.  Fermat similarly describes~$A$ as
\emph{inconnue} in a 22~october 1638 letter to Mersenne \cite[p.\;170,
175]{Ta2}.}
Fermat goes on to substitute~$A+E$ in place of~$A$, and writes:
\begin{quote}
Adaequentur, ut loquitur Diophantus, duo homogenea maximae aut minimae
aequalia \cite[p.\;133]{Ta1}
\end{quote}
(see Section~\ref{f14} for details and translation).  Following the
steps outlined in Section~\ref{s2} (and in more detail in
Section~\ref{s110}), Fermat concludes: 
\begin{quote}
The solution of this last equality will give the value of~$A$, which
known, the maximum or minimum will become known by repeating the
traces of the foregoing resolution. \cite[p.\;162]{Ma94}
\end{quote}
It is only at this stage in the algorithm that the \emph{unknown}~$A$
becomes a \emph{constant}~$A$, namely the desired point of maximum.%
\footnote{Thus Breger's claim to the effect that ``in the first text
on the method sent to Descartes, adaequare means equality''
\cite[p.\;21]{Br13} has no basis; see Section~\ref{s13}.}

\subsection{Original Latin}
\label{f14}

The original Latin of the passage on Diophantus cited in
Section~\ref{s21} reads as follows:
\begin{quote}
Adaequentur, ut loquitur Diophantus, duo homogenea maximae aut minimae
aequalia et, demptis communibus (quo peracto, homogenea omnia ex parte
alterutra ab~$E$ vel ipsius gradibus afficiuntur), etc.
\cite[p.\;133]{Ta1}.
\end{quote}
Mahoney translates this passage as follows: 
\begin{quote}
\emph{Adequate},%
\footnote{Mahoney is using \emph{adequate} as an English verb;
cf.\;note~\ref{f5}.}
as Diophantus says, the two homogeneous expressions equal to the
maximum or minimum and, having removed common terms (which done, all
homogeneous quantities on either side will contain~$E$ or degrees of
it), etc.''  \cite[p.\;162]{Ma94}
\end{quote}
Note that Fermat speaks of two homogeneous expressions equal to the
(i.e., representing the unknown) maximum.  He does not speak of two
homogeneous expressions equal to each other.%
\footnote{This is contrary to Breger's claim at \cite[p.\;28,
line\;3]{Br13}.  Breger reproduces only a misleadingly truncated
fragment of Fermat's Latin phrase, in footnote~50 there to support his
dubious claim.  Breger's fragment is ``duo homogenea maximae aut
minimae aequalia.''  However, Breger's fragment does not make sense
grammatically, because the words ``maximae/minimae" only make sense in
Fermat's full phrase as \emph{Datives} construed with ``aequalia",
namely ``two homogeneous terms equal to the maximum or the minimum".
The two homogeneous expressions are not equated but rather
\emph{adequated} by Fermat.  In the 1638 text mentioned in
note~\ref{f13c} Fermat writes (in old French) explicitly: ``comme
s'ils estoient esgaux, bien qu'en effect ils ne le soient pas''
\cite[p.\;74, lines\;15--16]{Fe22}.}

\subsection{Fermat's \emph{Sur la m\^eme m\'ethode}}
\label{tsi}

Fermat's Latin text \emph{Ad eamdem methodum} (item VI in \cite{Ta1})
starts at \cite[p.\;158]{Ta1}.  We will use Tannery's translation
\emph{Sur la m\^eme m\'ethode} for the reader's convenience.  Here we
find:
\begin{quote}
Nous consid\'erons en fait dans le plan d'une courbe quelconque deux
droites donn\'ees de position, dont on peut appeler l'une
\emph{diam\`etre}, l'autre \emph{ordonn\'ee}.  Nous supposons la
tangente d\'ej\`a trouv\'ee en un point donn\'e sur la courbe, et nous
consid\'erons par \emph{ad\'egalit\'e} la propri\'et\'e sp\'ecifique
de la courbe, non plus sur la courbe m\^eme, mais sur la tangente \`a
trouver.  En \'eliminant, suivant notre th\'eorie des maxima et
minima, les termes qui doivent l'\^etre, nous arrivons \`a une
\emph{\'egalit\'e} qui d\'etermine le point de rencontre de la
tangente avec le diam\`etre, par suite la tangente elle-m\^eme.%
\footnote{Latin original: ``Consideramus nempe in plano cujuslibet
curvae rectas duas positione datas, quarum altera diameter, si libeat,
altera applicata nuncupetur.  Deinde, jam inventam tangentem
supponentes ad datum in curva punctum, proprietatem specificam curvae,
non in curva amplius, sed in invenienda tangente, per adaequalitatem
consideramus et, elisis (quae monet doctrina de maxima et minima)
homogeneis, fit demum aequalitas quae punctum concurs\^us tangentis
cum diametro determinat, ideoque ipsam tangentem''
\cite[p.\;159]{Ta1}.}
\cite[p.\;141]{Ta3} (emphasis on \emph{diam\`etre} and
\emph{ordonn\'ee} in the 1896 source; emphasis on \emph{ad\'egalit\'e}
and \emph{\'egalit\'e} added)
\end{quote}
Here similarly, adequality refers to \emph{unknowns}, whereas equality
refers to the value of the \emph{constant} found.  Thus Fermat is
expressing distinct ideas when he uses the words \emph{adaequalitas}
and \emph{aequalitas}.  The characteristic pair ``adequality,
equality'' occurs here as elsewhere in that specific order; see
Remark~\ref{r110} on page \pageref{r110}.

\medskip\noindent [\textbf{Note added after publication:} After the
article was published online (see
\url{http://dx.doi.org/10.1007/s10699-017-9542-y}) at
\emph{Foundations of Science}, we came across an important piece of
evidence concerning Fermat's Toulouse friend Lalouv\`ere, jesuit,
mathematician, and \emph{censor}.  Indeed, Antonella Romano writes:
``Un premier constat regarde les j\'esuites confront\'es \`a la
censure: sur tous les cas du Fondo Gesuitico qui concernent la France,
deux professeurs seulement appartiennent \`a la liste \'etablie dans
le cadre de cet ouvrage, B. Labarthe et V. L\'eotaud.  Si d'autres
math\'ematiciens y apparaissent, c'est tout aussi exceptionnellement,
et au titre de censeur, comme Antoine Lalouv\`ere$^{114}$''
\cite[p.\;512]{Ro99}.  Footnote 114 there reads: ``C'est lui qui porte
un jugement n\'egatif sur l'ouvrage de B. Labarthe.''  Thus,
Lalouv\`ere in his capacity as a censor sank at least one book, namely
that by his fellow jesuit Labarthe.]

\section{Reductive readings}

Reductive tendencies in modern historiography of mathematics were
analyzed in \cite{12d}; see also Section~\ref{s27}.  Some Fermat
scholars have argued recently that \emph{adequality} meant
\emph{equality} or \emph{setting equal}, that Fermat's procedure was
purely algebraic, and that no idea of approximation was involved.

\subsection{Reductive readings of adequality: Breger}
\label{s13}

Breger's reading along these lines was critically analyzed in
\cite{13e}.  That such scholars are in a distinct minority emerges
from the following comment of Breger's:
\begin{quote}
Although there are considerable differences between various
interpretations, there seems to be a common dogma, to which all
interpreters, as far as I know, agree, namely: \emph{Fermat uses the
word `adaequare' in the sense of `to be approximately equal' or `to be
pseudo-equal' or `to be counterfactually equal'} \cite[p.\;194]{Br94}.
(emphasis in the original)
\end{quote}
Why is it that ``all interpreters,'' as reported by Breger, agree that
a relation more general than equality and exploiting in some way
\emph{in}equality is involved in Fermat's method?  Possibly it is
because Fermat himself said so in a letter to Mersenne:
\begin{quote}
Cette comparaison par \emph{ad\'egalit\'e} produit deux termes
\emph{in\'egaux} qui enfin produisent l'\'egalit\'e (selon ma
m\'ethode), qui nous donne la solution de la question
\cite[p.\;137]{Fe38} (emphasis on \emph{ad\'egalit\'e} in the
original; emphasis on \emph{in\'egaux} added)
\end{quote} 
(see Section~\ref{s22} for details).  Similar remarks apply to
Fermat's later missive to Descartes mentioning \emph{ad\'equation} and
\emph{in}equalities involved; see \cite[p.\;155]{Ta2}.  Meanwhile,
Breger claims that
\begin{quote}
two homogeneous quantities that are each equal to the maximum or
minimum are obviously equal to each other and not only approximately
equal \cite[p.\;21, note\;6]{Br13}.
\end{quote}
A similar claim had already appeared in \cite[p.\;195]{Br94}.
Breger's inference is based on assuming that the symbol~$A$ appearing
in Fermat's argument (see Section~\ref{s21}) is a \emph{constant}
rather than an \emph{unknown} throughout the argument, but~$A$ is
clearly an unknown at the outset of the argument and only becomes a
constant at the end, as analyzed in Sections~\ref{s21} and~\ref{tsi}.
Thus Breger's inference is based on a confusion of unknowns and
constants.

\subsection{Fermat's letter; \emph{enfin}}
\label{s22}

For the sake of completeness we reproduce the full passage from
Fermat's letter to Mersenne cited in Section~\ref{s13}, which uses the
term \emph{ad\'egalit\'e} twice:

\begin{quote}
6. Outre le papier envoy\'e \`a R(oberval) et P(ascal), pour
suppl\'eer \`a ce qu'il y a de trop concis, il faut que M.\;Descartes
sache, qu'apr\`es avoir tir\'e la parall\`ele qui concourt avec la
tangente et avec l'axe ou diam\`etre des lignes courbes, je lui donne
premi\`erement le nom qu'elle doit avoir comme ayant un de ses points
dans la tangente, ce qui se fait par la r\`egle des proportions qui se
tire des deux triangles semblables.  Apr\`es avoir donn\'e le nom,
tant \`a notre parall\`ele qu'\`a tous les autres termes de la
question, tout de m\^eme qu'en la parabole, je consid\`ere derechef
cette parall\`ele, comme si le point qu'elle a dans la tangente
\'etoit en effet en la ligne courbe, et suivant la propri\'et\'e
sp\'ecifique de la ligne courbe, je compare cette parall\`ele par
\emph{ad\'egalit\'e} avec l'autre parall\`ele tir\'ee du point donn\'e
\`a l'axe ou diam\`etre de la ligne courbe.  Cette comparaison par
\emph{ad\'egalit\'e} produit deux termes in\'egaux qui \emph{enfin}
produisent l'\'egalit\'e (selon ma m\'ethode), qui nous donne la
solution de la question.  \cite[p.\;137]{Fe38} (emphasis on
\emph{enfin} added)
\end{quote}
The adverb \emph{enfin} indicates that something \emph{novel} occurs
at the end of the argument that hasn't been the case earlier: namely,
\emph{equality}.  In the same letter Fermat makes (only a half-)joking
reference to ``ma petite guerre contre M.\;Descartes'' (ibid., item
5); one has to assume that under the circumstances Fermat would have
endeavored to provide his best available explanation of the method.

\subsection{Did Tannery tamper with Fermat's manuscript?}
\label{f9b}

A facsimile of a page from Fermat's handwritten manuscript starting
with the words \emph{Doctrinam tangentium} is reproduced between pages
xviii and\;xix in \cite{Ta1}.  Here Fermat used the term
\emph{adaequalitatem} once and, two lines later, the term
\emph{aequalitas} once.  The characteristic pair ``adequality,
equality'' (see Remark~\ref{r110}) occurs numerous times throughout
Fermat's writings on the method.  Tannery's transcription of Fermat's
page appears at \cite[pp.\;158--159]{Ta1} as part of a text entitled
\emph{Ad eamdem methodum} (item VI in Tannery's edition).  The terms
in question appear at lines\;14 and 16 of page 159 of the
transcription.  An editorial note at \cite[p.\;426]{Ta1} deals with
Tannery's transcription.  The note concerning line 14 states the
following: ``14\;adaequalitatem] aequalitatem \emph{Va}''
\cite[p.\;426]{Ta1}.  Here the abbreviation \emph{Va} refers to the
edition published by Samuel Fermat in 1679 entitled \emph{Varia Opera}
\cite{Fe79}, as explained earlier on the same page\;426.  The note
indicates that Tannery corrected an error that crept into the
\emph{Varia Opera} edition of 1679 (which replaced Fermat's term
\emph{adaequalitatem} by \emph{aequalitatem}).

We see therefore that Tannery did not attempt to ``repair'' Fermat's
handwritten manuscript.  Tannery's transcription, faithful to the
original, uses \emph{adaequalitatem} once at line\;14 there, and
\emph{aequalitas} once at line\;16.%
\footnote{For the sake of completeness we reproduce lines 14, 15, 16,
and 17 exactly as they appear in \cite[p.\;159]{Ta1}: \hfill\hfill
\linebreak (14) \hbox{in invenienda tangente, per adaequalitatem
consideramus et, elisis} \hfill\hfill \linebreak (15) \hbox{(quae
monet doctrina de maxima et minima) homogeneis, fit demum}
\hfill\hfill \linebreak (16) \hbox{aequalitas quae punctum concurs\^us
tangentis cum diametro determinat,} \hfill\hfill \linebreak (17)
\hbox{ideoque ipsam tangentem.}\hfill\hfill}
Similarly, Tannery's French translation \cite[p.\;141]{Ta3}
(reproduced in Section~\ref{tsi} above) uses \emph{ad\'egalit\'e} once
and \emph{\'egalit\'e} once.

Breger writes:
\begin{quote}
The facsimile (Fermat 1891, after XVIII) gives the impression that the
manuscript was not written in a hurry; there are few additions and
deletions, and the handwriting is fair.  In two passages of the
manuscript, Fermat expresses the same idea in nearly identical words,
but in [sic] the first time he uses the word adaequalitas, whereas in
the second passage the word aequalitas is used (Fermat 1891, 159, 162,
426) \cite[p.\;195]{Br94}.
\end{quote}
Breger goes on to claim that 
\begin{quote}
the editors of the Oeuvres deemed it necessary to repair Fermat's
supposed confusion by changing the aequalitas into another
adaequalitas [but] they nevertheless indicated their change in the
critical apparatus (Fermat 1891, 426) \cite[pp.\;195--196]{Br94}.
\end{quote}
This refers to an edit on page\;162 at line\;22, from
\emph{aequalitas} to \emph{adaequalitas}.  We will discuss Tannery's
edit in more detail in Section~\ref{s24}.

\subsection{Breger's judgment of Tannery}
\label{s24}

Breger juxtaposes his judgment of a change in the 1679 edition with
his judgment of a change in Tannery's 1891 edition in the following
terms:
\begin{quote}
The editors of (Fermat 1679) present a version which has been changed
the other way round: They give twice aequalitas, whereas (Fermat 1891)
gives twice adaequalitas.  \cite[p.\;196]{Br94}.
\end{quote}
Breger's tone is even more strident in his 2013 article:
\begin{quote}
In the only manuscript on the method of maxima, minima and tangents
that has been handed down to us in Fermat's own handwriting,
aequalitas and adaequalitas are used \emph{interchangeably}.
Unfortunately, the editors of the Fermat edition have interfered with
the text to support their own interpretation of the method and they
have declared the text in Fermat's handwriting to be flawed.
\cite[p.\;22]{Br13} (emphasis added)
\end{quote}
Breger's claim that Fermat used aequalitas and adaequalitas
\emph{interchangeably} is unsupported by evidence.  In fact, Fermat
uses them in distinct senses on the page in Fermat's handwriting
reproduced between xviii and xix in \cite{Ta1}, and accurately
transcribed by Tannery on page\;159.  The term \emph{adequality}
always occurs in tandem with \emph{equality} with the former denoting
a comparison of expressions involving \emph{unknowns} and the latter,
an expression for the \emph{constant} providing the final answer.
Whenever Fermat presents a detailed solution, the characteristic pair
``adequality, equality'' invariably appears; see Remark~\ref{r110} on
page \pageref{r110}.

Breger's comment at \cite[p.\;196]{Br94} cited above is equally
critical of the changes by the 1679 editors and by the 1891 editors.
However, the change made by the 1679 edition at \cite[p.\;69]{Fe79}
was an erroneous replacement of \emph{adaequalitatem} by
\emph{aequalitatem}, contrary both to the version found in Fermat's
handwritten manuscript (see Section~\ref{f9b}) and to Fermat's
description of his method in the text sent to Descartes and appearing
at \cite[p.\;133]{Ta1}; see Section~\ref{s21}.

Namely, the edit by the 1679 editors disrupted the characteristic pair
``adequality, equality.''  Such a pair appears in Fermat's detailed
presentation of a solution on numerous occasions, including Fermat's
usage recorded on page 159 (of the 1891 edition) and Fermat's usage
recorded on page 163 (of the 1891 edition) on the cycloid (``curva
Domini de Roberval''); see Section~\ref{s2}.  Therefore the change on
page\;69 of the 1679 edition constituted an editorial error.

Meanwhile, the change made by Tannery in a passage on page\;162,
line~22 replacing \emph{aequalitas} by \emph{adaequalitas} concerns a
\emph{stand-alone} occurrence of the term (rather than as part of the
characteristic pair) in a brief concluding passage, which we reproduce
in translation for the reader's convenience:
\begin{quote}
Enfin, ce qui est le point important, aux arcs de courbes on peut
substituer les longueurs correspondantes des tangentes d\'ej\`a
trouvées, et arriver \`a l'\emph{ad\'egalit\'e},%
\footnote{Some of the extant texts have \emph{\'egalit\'e} instead.}
comme nous l'avons indiqu\'e : on satisfera ainsi facilement \`a la
question.%
\footnote{Latin original: ``et demum (quod operae pretium est)
portiones tangentium jam inventarum pro portionibus curvae ipsis
subjacentis sumantur, et procedat adaequalitas ut supra monuimus:
proposito nullo negotio satisfiet.'' \cite[p.\;162]{Ta1}}
\cite[p.\;143--144]{Ta3}
\end{quote}
Some of the extant texts for this passage have \emph{aequalitas} and
others have \emph{adaequalitas}, and editor Tannery chose to go with
the latter.  Tannery appropriately indicated the existence of
alternative versions in an editorial note at \cite[p.\;426]{Ta1}
concerning page~162, line~22.  Either term would fit here, and the
choice does not affect the interpretation of the method.

Breger implies that the 1679 and the 1891 editors were equally guilty
of tampering with Fermat's original but Breger's judgment of Tannery's
scholarship is biased by Breger's own quest to account for adequality
in terms of the \emph{double root} idea; see Remark~\ref{r23} on page
\pageref{r23}.

\subsection{Cherry-picked numerology}

The numerology of Breger's comment to the effect that ``The editors of
(Fermat 1679) present a version which has been changed the other way
round: They give twice aequalitas, whereas (Fermat 1891) gives twice
adaequalitas'' \cite[p.\;196]{Br94} is surprising if not to say
incomprehensible.  Where does the number~$2$ as in ``[t]hey give twice
aequalitas, whereas (Fermat 1891) gives twice adaequalitas'' come
from?

The handwritten manuscript transcribed at \cite[p.\;158--167]{Ta1}
uses the term \emph{adaequalitatem} (or \emph{adaequetur},
\emph{adaequalitas}, \emph{adaequari}) no fewer than 10 times.%
\footnote{More precisely, the term occurs once on each of the pages
159, 160, 161, three times on 162, twice on 163, and twice on 164.}
Even if one takes off the one occurrence on page 162 that Breger is
unhappy about, that still leaves us with nine occurrences of the term
in Fermat's handwritten manuscript.  

Breger's number~$2$ seems to be cherry-picked so as to attack
Tannery's editing by comparing it to that of the 1679 edition.

\subsection{Breger's wrong move}
\label{s111}

We will now illustrate how a scholar working in an
infinitesimal-\emph{frei} conceptual framework misses important
features of Fermat's method.  Let us consider Breger's transcription
of Fermat's argument presented in Section~\ref{s110}.  Breger writes:
\begin{quote}
The maximum of~$BA-A^2$ is to be found.  Fermat puts
\begin{equation}
\label{e13}
BA-A^2=B(A+E)-(A+E)^2
\end{equation}
\begin{equation}
\label{e14}
2AE+E^2-BE=0
\end{equation}
\begin{equation}
\label{e15}
E(2A+E-B)=0
\end{equation}
\ldots{} The procedure consists of three steps: Firstly, two
functional expressions for the extreme value are equated; then one
divides by~$E$ \ldots{} in the third step~$E$ is put equal to zero.
\cite[p.\;27]{Br13} (equation numbers added)
\end{quote}
Breger mentions the steps involving the division by~$E$ and
setting~$E$ equal to zero, but does not include the resulting
relations, so we will do it for him, following the style of his
formulas~\eqref{e13}, \eqref{e14}, \eqref{e15}:
\begin{equation}
\label{e16}
2A+E-B=0
\end{equation}
and then
\begin{equation}
\label{e17}
2A-B=0.
\end{equation}
What is striking about Breger's paraphrase is its insensitivity to the
details of Fermat's mathematical presentation.  Breger's
relation~\eqref{e13} is true to Fermat but already Breger's next
relation \eqref{e14} betrays Fermat.

\begin{remark}
As is evident from our summary in Section~\ref{s110}, \emph{Fermat
never moved the terms to the left-hand side as Breger did}.  In fact,
Fermat wrote: ``aequentur sane \ldots negata affirmatis''
\cite[p.\;134]{Ta1} which Mahoney translates as ``equate \ldots{} the
negative terms to the positive'' \cite[p.\;162]{Ma94}; see
Remark~\ref{r19}.  Breger's wrong move makes Fermat's procedures
appear unclear or even confused, as we now show.
\end{remark}

The transition from Breger's relation \eqref{e14} to relation
\eqref{e15} is an unproblematic algebraic transformation using the
distributive law.  However, Breger's next step is problematic and in
fact untenable.

\begin{remark}
The transition from Breger's relation \eqref{e15} to Breger's relation
\eqref{e16} can be justified neither in terms of the
relation~$\approx$ nor in terms of the relation~$\adequal$; see
Remark~\ref{r16}.  Nor did Fermat use~$0$ appearing on Breger's
right-hand side.
\end{remark}

If one wishes to interpret relation \eqref{e16} as true equality as
Breger does, the passage from \eqref{e16} to \eqref{e17} becomes
incomprehensible.  More precisely, it forces~$E=0$ and invalidates any
possible passage from \eqref{e15} to \eqref{e16}.

Breger is ``shocked'' at \cite[pp.\;25--26]{Br13} by H\'erigone's
\emph{Cursus mathematicus} which if interpreted literally appears to
present a version of Fermat's method where one patently divides by
zero to obtain the final result.  Yet Breger's own flawed paraphrase
at \cite[p.\;27]{Br13} similarly relies on~$E$ being first nonzero and
then zero.

Breger declares that ``for the sake of clarity a~$\ldots=0$ is
occasionally used, although Fermat himself avoids using it''
\cite[p.\;29]{Br13} but what Breger introduced by tampering with
Fermat's presentation is not clarity but rather confusion.

In 1642, H\'erigone presented the same example while retaining
Fermat's balanced presentation:
\[
2ae+e^2 \quad  2|2 \quad eb,
\]
followed by
\[
2a+ e \quad 2|2 \quad b,
\]
and finally
\[
2a \quad 2|2 \quad b
\]
(here~$2|2$ denotes an equality sign) as found in \cite[p.\;60]{He42}.
To his credit, \cite[p.\;18]{Ba11} treats the same example without
distorting Fermat's method the way Breger does by moving the terms to
the left-hand side.

\begin{remark}
\label{r23}
Why did Breger move the terms to the left-hand side so as to obtain
the formula~$2AE+E^2-BE=0$ as in \eqref{e14}, whereas neither Fermat
nor his popularizer H\'erigone did?  The polynomial~$2AE+E^2-BE$ has a
\emph{double root} when~$A=B/2$.  Breger seeks a uniform
interpretation of Fermat's method in terms of the \emph{double root}
idea.  But this is not Fermat's way.  The double root idea is present
in Fermat's work but the method of adequality is more general and some
of its applications resist Breger's reductive paraphrase.  Imagine
what \mbox{Unguru} might think of such a left-handed move.
\end{remark}

\subsection{Triumvirate bias}
\label{s27}

Carl Boyer referred to Cantor, Dedekind and Weierstrass (CDW) as ``the
great triumvirate'' in \cite{Bo49}.  A hagiographic attitude toward
CDW tends to go hand-in-hand with a distrust of mathematics that CDW
were unable to formalize, such as infinitesimals.

Thus, volume 4 of Fermat's collected works contains the following
dismissive comments by A.\;Aubry concerning Kepler and Cavalieri:
\begin{quote}
[Fermat] semble admettre ainsi, comme Kepler et Cavalieri, l'existence
r\'eelle des infiniment petits, et ce principe que deux quantit\'es ne
diff\'erant que d'un infiniment petit sont \'egales; tandis qu'au
contraire Descartes se fonde sur la th\'eorie des racines \'egales,
c'est-\`a-dire, au fond, sur celle des limites''%
\footnote{From an editorial comment on page 145 in the same volume it
emerges that Aubry developed his remarks (in Note XXV) especially for
the 1912 volume.}
\cite[p.\;225]{Ta4}.
\end{quote}
Aubry clearly has little confidence in ``really existing
infinitesimals'' or for that matter in a relation of generalized
equality between quantities differing by an infinitesimal.

Scholars trained in the traditional Weierstrassian framework based
upon the real continuum tend to privilege such a framework over other
modern frameworks even when the latter are procedurally closer to the
pioneering epoch of analysis.  Breger frequently mentions the
following two points in close proximity to each other:
\begin{enumerate}
\item
the concept of approximate equality, and
\item
the issue of Fermat being unclear, confused, or contradictory.
\end{enumerate}
Such juxtaposition suggests the existence of an implied connection
between the two points in Breger's mind.  Thus, he writes:

\subsubsection{}
``The commonly accepted interpretations of Fermat's method of extreme
values tell us that this is a curious method, based on an approximate
equality and burdened with several \emph{contradictions} within
Fermat's writings.''  \cite[p.\;139]{Br94} (emphasis added)

\subsubsection{}
``Although there are considerable differences between various
interpretations, there seems to be a common dogma, to which all
interpreters, as far as I know, agree, namely: Fermat uses the word
`adaequare' in the sense of `to be approximately equal' \ldots{}
Usually, the explicit or tacit assumption is made that Fermat himself
was somewhat \emph{confused}.''  \cite[p.\;194]{Br94} (emphasis added)

\subsubsection{}
``In two passages of the manuscript, Fermat expresses the same idea in
nearly identical words, but in [sic] the first time he uses the word
adaequalitas, whereas in the second passage the word aequalitas is
used (Fermat 1891, 159, 162, 426). If the dogma of the usual
interpretation is assumed, then this is another example of Fermat's
supposed \emph{confusion}.''  \cite[p.\;195]{Br94} (emphasis added)

\subsubsection{}
``Looking at these oddities, we finally arrive at an alternative and
at a warning.  The alternative is: Either Fermat was pretty
\emph{confused - far more confused} than a mathematician presenting
one of his central ideas is expected to be - or something is
fundamentally wrong with our understanding of Fermat's method of
extreme values.''  \cite[p.\;196]{Br94} (emphasis added)

\medskip
Breger's tone becomes progressively more strident:

\subsubsection{}
``There is no approximate equality involved in Fermat's insight and in
his application of the Diophantus passage.  Thus I would like to
conclude that these words were written by someone who did not really
understand the meaning of Fermat's reference to Diophantus.  \ldots{}
It seems hard to believe that Fermat could have been so
\emph{confused} as to write \emph{this rubbish}.''
\cite[p.\;210]{Br94} (emphasis added)

\subsubsection{}
``The text's \emph{contradictoriness} cannot be remedied by using
words such as `approxiately equal, conterfactually equal, pseudoequal,
limit process, infinite approximation, infinitesimal' in their
interpretation.  If one finds it hard to believe that Fermat had
repeatedly got caught up in \emph{contradictions}, etc.''
\cite[p.\;20]{Br13} (emphasis added)

\medskip
The following passage is particularly revealing of Breger's assumption
that an approximate equality cannot be part of a proper \emph{proof}:

\subsubsection{}
``[Fermat] knew that an approximate equation was different from a
\emph{proven} convergence (even though he did not yet have the word
convergence), and it would be most strange if he had not distinguished
the one from the other."  \cite[p.\;23]{Br13} (emphasis added)

\medskip
These passages juxtaposing the idea of approximation and the idea of
confusion reveal Breger's manifest belief that viewing Fermat as using
any version of \emph{approximate} equality is inevitably tied in with
viewing Fermat as using arguments that were only \emph{approximately}
correct and thus vague or confused.

Scholars trained in the Weierstrassian school tend to assume that any
relation of \emph{approximate equality} necessarily constitutes a
vague or even confused concept from a strictly mathematical viewpoint.
This type of bias results from the fact that no such relation is
available in the context of the real continuum.  However, such
relations are indeed available in infinitesimal-enriched frameworks;
see Section~\ref{s110}.

\subsection{Breger's equivocation}
\label{s28}

Breger's interpretation is based on an equivocation on the meaning of
Fermat's term.  Breger seeks to assign two distinct meanings to the
term \emph{adequality}:
\begin{enumerate}
\item
\label{itemi}
simple coincidence and/or equivalence;
\item
\label{itemii}
``setting equal" in the sense elaborated by Breger in his paper.
\end{enumerate}
Breger's claim that Fermat used the terms \emph{equality} and
\emph{adequality} interchangeably is only plausible if based on the
sense\;\eqref{itemi}.  However, Breger himself explains the term
adequality in the sense\;\eqref{itemii} relying on a theorem he states
at \cite[p.\;29]{Br13} exploiting expressions like~$(f(x,g(x))$ and
bearing a close resemblance to an implicit function theorem; Barner
explicitly admits using such a 19th century result in
\cite[p.\;17]{Ba11}.

\subsection{Fermat versus Breger}

We bring the following comparison to our readers' attention.

\begin{framed}
\setlength{\columnsep}{10pt}
  \begin{multicols}{2}
\noindent\textbf{Fermat to Mersenne:} `Cette comparaison par
\emph{ad\'egalit\'e} produit deux termes in\'egaux qui enfin
produisent l'\'egalit\'e (selon ma m\'ethode), qui nous donne la
solution de la question.' \cite[p.\;137]{Fe38} \vfill\eject
\noindent\textbf{Breger:} `Fermat's method is not to be located in the
region of approximations\ldots It has nothing to do with an
approximate equality, and thus lacks any basis for assumptions about
limit processes or infinitesimals.'  \cite[p.40]{Br13}
  \end{multicols}
\end{framed}

\subsection{A self-inflicted textual difficulty}

In his text \emph{Ad eamdem methodum} (item III in \cite{Ta1}), Fermat
speaks of \emph{comparison by adequality}:
\begin{quote}
Comparanda sunt ergo homogenea notata signo $+ cum$ iis quae notantur
signo $-, et$ iterare \emph{comparationem [adaequalitatem]} oportet
inter $B$ in $Eq$. + $B$ in $A$ in $E$ bis ex una parte, et $A$ in
$Eq$. ter + $Aq$. in $E$ ter + $Ec$. ex altera.
\cite[pp.\;140--141]{Ta1} (emphasis added)
\end{quote}
Note that the brackets around \emph{adaequalitatem} are Tannery's
(they are not to be found in Samuel Fermat's edition
\cite[p.\;66]{Fe79}).  In footnote\;1 on page 141, Tannery notes that
the text may not be authentic since Fermat should have written only
one of the two words, which according to Tannery Fermat used as
synonyms:
\begin{quote}
Le texte v\'eritable est douteux: Fermat n'a d\^u \'ecrire que l'un
des deux mots, \emph{comparationem} ou \emph{adaequalitatem}, qu'il
employait comme synonymes; l'autre serait une glose du copiste ou du
possesseur de l'original.  M\^eme remarque pour \emph{comparatio} et
\emph{adaequalitas}, quatre lignes plus bas.  \cite[p.\;141,
note\;1]{Ta1} (emphasis in the original)
\end{quote}
But in de Waard's volume we find Fermat's French translation of this
text, where Fermat again speaks of \emph{comparison by adequality}:
\begin{quote}
Il faut donc comparer les homog\`enes qui sont marquez du signe $+$
avec ceux qui sont marquez du signe $-$, et faire derechef comparaison
\emph{adaequalitatem} entre $B$ in $Eq$. $+$ $B$ in $A$ in $E$ bis
d'un cost\'e, et $A$ in $Eq$. ter $+$ $Aq$. in $E$ ter $+$ $Ec$. de
l'autre.  \cite[p.\;75]{Fe22} (emphasis in the original)
\end{quote}
This time it is editor de Waard's turn to add a footnote\;1 concerning
alleged textual difficulties, without however adding brackets around
\emph{adaequalitatem}.

In fact there is no difficulty here since Fermat is using
\emph{comparison} in a generic sense of the word, whereas he is using
\emph{adequality} as the specific technical term exploited in his
method, as he writes:
\begin{quote}
\ldots{} et i'ay appel\'e en mon escrit latin cette sorte de
comparaison \emph{adeaqualitatem} \ldots \cite[p.\;74]{Fe22}
\end{quote}
Moreover, in his remarks about Diophantus in \cite[p.\;140]{Ta1}
\cite[p.\;74]{Fe22}, Fermat also uses both words but here Tannery and,
following him, de Waard do not protest.

It seems that recognizing the possibility that Fermat fully intended
to use both words (\emph{comparison} by \emph{adequality}) involves
imagining a mathematically precise notion of generalized equality
different from the generic sense of comparison, an idea apparently not
easy in a world of post-Weierstrassian historiography.

\subsection{Reductive readings of adequality: Felgner}

Ulrich Felgner for his part focuses on examples reducible to a
\emph{polynomial} which is positive in a neighborhood of a zero at
$h=0$, and seeks to account for all these examples based on his
interpretation of \parisotes{} in \cite{Fe16}.  However, he is unable
to account for Fermat's treatment of transcendental curves like the
cycloid, a difficulty one would expect in light of Weil's thesis; see
Section~\ref{s13b}.

\subsection{Reductive readings of adequality: Barner}

In an odd twist of Fermat scholarship, Klaus Barner appears to hold
that he understood Fermat's method better than\ldots{} Fermat himself.
Barner claims that Fermat's attempt to explain his method of
calculating the tangent at a given point of a curve is somewhat
confusing\;(!), that Fermat actually calculates with a \emph{secant
line} which converges to the tangent as the limit line, and that
Fermat's problem is that he calculates with secants \emph{without
being aware of that}; cf.~\cite[pp.\;32--36]{Ba11}.

The problem with Barner's approach is not merely that, as he seems to
acknowledge, there is no evidence for his \emph{secant} hypothesis in
Fermat's texts, but also that the secant approach is known as Jean
Beaugrand's rather than Fermat's, as analyzed
in~\cite[pp.\;64--65]{Strom} and noted in \cite{13e}.  In fact,
Barner's \emph{secant} reading is not new.  Cifoletti writes:
\begin{quote}
[Zeuthen] fait r\'ef\'erence \`a une s\'ecante, (qui approximerait la
tangente).  Cela est surprenant, puisqu'il n'y a aucun texte de Fermat
sur lequel appuyer cette interpretation.  Bien entendu, la version de
Beaugrand suit cette approche \cite[p.\;47]{Ci90}.
\end{quote}
The Zeuthen--Barner reading of Fermat's method has not been accepted
by modern Fermat scholars.

\section{Historical setting}

A review of the historical setting at the time of Fermat's activity in
mathematics will necessarily involve his professional activities as
member of the Parliament of Toulouse.

\subsection{Salient historical points}

The 11 relevant points are as follows.

\subsubsection{Parliament}
The modern connotations of the term \emph{Parliament} should not
mislead one into thinking about 17th-century Toulouse in terms of
modern political concepts.  In the 17th century the doctrine of the
separation of powers was not even conceived, let alone implemented, in
the Western world.  The Parliament of Toulouse was the local judicial,
legislative, and executive power rolled into one, and subordinate only
to the king of France.

\subsubsection{Venality}
Fermat, like just about everybody else who served in the
Parliament, had to pay a hefty fee%
\footnote{Of 43500 livres \cite[p.\;16]{Ma94} in Fermat's case.  At
the time all public functions were sold under certain conditions
\cite{Mo71}.}
for the privilege.

\subsubsection{Counterreformation}
This was the period of counter-reformation, when the
interdenominational strife was at its height.  Fermat was catholic but
his parents had some protestant sympathies.%
\footnote{\label{f8}Fermat's probable mother, Claire \emph{n\'ee} de
Long, came from a Huguenot family from Montauban, a Huguenot
stronghold \cite{Ga01}.  His wife (and distant cousin) Louyse was also
a de Long (\cite{Mo17} prefers the spelling ``Louise'' while
\emph{Codicille au testament de Pierre de Fermat} uses the spelling
``Louyse'' \cite[p.\;347]{Ch67}).  In addition, Fermat had close
protestant friends in the \emph{Chambre d'\'Edit} at Castres, namely
Pierre Saporta and Jacques de Ranchin \cite{Ch67}.  The claim that
Fermat ``was baptized (and most probably born) on 20\;August\;1601 to
Dominique Fermat \ldots{} and his wife Claire, \emph{n\'ee} de Long"
\cite[p.\;15]{Ma94} is incorrect, since we know that in 1603 Dominique
was still married to Fran\c coise Cazeneuve \cite[p.\;172]{Sp08}.
Claire de Long's grandfather Jean de l'Hospital was a Huguenot
expelled from the Parliament for religious reasons; see
Section~\ref{s41}.}
Fermat's ``mixed'' background was to influence the trajectory of his
professional activities as well.

\subsubsection{Castres}
Whereas Fermat was based at Toulouse, he paid frequent visits to the
city of Castres starting in 1638 \cite[p.\;340]{Ch67} with protestant
leanings, in his function as member of the \emph{Chambre de l'\'Edit}%
\footnote{This refers to the \emph{Edit de Nantes} (1598) or the Edict
of Nantes, where a compromise was reached between catholics and
protestants that held until its abolition in 1685 by the \emph{l'Etat,
c'est Moi} (The \emph{State}, it is \emph{Myself}) king who apparently
found no room within either his catholic \emph{self} or his
\emph{State} for protestantism.}
of the Parliament of Toulouse.  This was done in his capacity as
negotiator of catholic/protestant tensions; see Section~\ref{s6}.

\subsubsection{Eucharist and atomism}
One of the stickiest points of the catho\-lic/protestant disagreement
concerns the doctrine of the eucharist (see Section~\ref{s5}).  What
is relevant as far as Fermat's mathematics is concerned is that the
theory of atomism and kindred doctrines were viewed by some catholic
theologians as a threat to their interpretation of the doctrine of the
eucharist; see Section~\ref{s63} on Trent 13.2.

\subsubsection{Atomism and indivisibles}
Perhaps through careless choice of wording by Cavalieri in his
correspondence with Galileo and others, \emph{atomism} was closely
associated with mathematical \emph{indivisibles}.%
\footnote{Lalouv\`ere specifically complained about this choice of
terminology, as noted in\;\cite[p.\;269]{De15}.}
Cavalieri belonged to the order of the jesuats, which was an older
order than the jesuit order and something of a rival.  Jesuit Guldin
attacked Cavalieri both for allegedly plagiarizing Kepler and for an
alleged fundamental incoherence of Cavalieri's technique exploiting
indivisibles.

\subsubsection{Indivisibles and infinitesimals}

Modern scholars distinguish carefully between indivisibles and
infinitesimals but in the 17th century the situation was less
clearcut.  Similarly, there was no clear-cut distinction at the time
between physical indivisibles and mathematical indivisibles.  The
distinction commonly made today follows \cite{Ko54} (among others).
The distinction is that indivisibles are codimension-one entities%
\footnote{This means that the dimension of the entity is one less than
the dimension of the ambient figure.}
whereas infinitesimals are of the same dimension as the entity they
make up (curve, planar region, etc).  The term \emph{infinitesimal}
was coined by either Mercator or Leibniz in the 1670s; see
\cite[p.\;63]{Le99}.  Similarly, there was no clear-cut distinction at
the time between physical indivisibles and mathematical indivisibles.
Festa observes:
\begin{quote}
L'atomisme \'etait m\^el\'e \`a ces discussions car les j\'esuites,
ceux du Collegio Romano surtout, ne faisaient pas grande distinction
entre `milieux g\'eom\'etriques' et `milieux phy\-siques' lorsqu'il
s'agissait de manifester leur opposition \`aà la notion de
`discontinuit\'e' \cite[p.\;103]{Fe91}.
\end{quote}
There was apparently a close connection, in the eyes of the catholic
clergy, among atoms, indivisibles, and the infinitely small.

\subsubsection{Anti-infinitesimal bans}
Festa documents a series of bans issued by the jesuit establishment
against indivisibles, including one in 1632, the year Galileo received
a summons to stand trial over heliocentrism; see \cite{Fe90} and
\cite[p.\;207]{Fe92}.  Only a few years earlier (around 1629) Fermat
developed the technique of adequality at Bordeaux.

\subsubsection{Degli Angeli} 

Cavalieri's student (or student's student via Torricelli) Stefano
degli Angeli was also a member of the order of the jesuats.  He
published many books on indivisibles, and also published several
spirited, and sometimes strongly worded, defenses of indivisibles
against attacks by jesuit scholars like Paul Guldin, Mario Bettini,
and Andr\'e Tacquet \cite[p.\;291]{Re87}.  In 1668 the order of the
jesuats was banned by papal brief, providing essentially no
explanation.  Degli Angeli lived many years afterwards and published
many additional books, but he never published another word on
indivisibles.

\subsubsection{Gregory}

James Gregory was degli Angeli's student during his stay in Padua from
1664 until 1668.  Gregory's books were suppressed in Venice around
1668 or 1669; see Section~\ref{s36}.

\subsubsection{Lalouv\`ere}

Fermat was in contact with the jesuit Antoine Lalou\-v\`ere at
Toulouse who was interested in mathematics (though the latter was
absent from Toulouse between 1632 and 1640).  Starting in 1658 Fermat
helped Lalouv\`ere with part of the latter's submission to the contest
of the \emph{roulette} (to determine the area of regions defined in
terms of the cycloid curve) launched anonymously by Pascal (and
eventually won by Pascal).  Via Lalouv\`ere and other channels, Fermat
would have been aware of the jesuits' position on indivisibles and
related doctrines.

\subsection{The 17th century scientific context}
\label{s3}

Most scholars acknowledge that infinitesimal analysis was a natural
outgrowth of the techniques of indivisibles as developed by Galileo's
student Cavalieri, and in fact Galileo's own work may be closer to the
infinitesimal techniques of their contemporary Kepler.

Earlier authors used other terms.  Thus, Kepler referred to a point of
the circumference of the circle as a \emph{basis quantulacunque}; to a
straight line, as an \emph{areola}; a rotating axis moved
\emph{minimum} about itself, etc.  Yet he meant what we would call
today the infinitely small.  Galileo spoke of \emph{non-quanta}
\cite{Ba14}, \cite{Ba14b}.  Other authors referred either to
\emph{infinite parva} or various translations thereof of type
\emph{infinitely small}, \emph{infiniment petit}, etc.

Ingegno discusses 17th century atomism, vacuum, and the related
theological difficulties in the following terms: 
\begin{quote}
Le vide renvoyait aux positions atomistiques qui, dans l'\'ecole
galil\'eenne, s'appuyaient sur l'enseignement du ma\^\i tre.  Il
s'agit en particulier de l'enseignement dispens\'e par Galil\'ee
d\'ej\`a \^ag\'e\ldots Les discussions europ\'eennes sur le vide ne
pouvaient qu'accentuer les critiques antiaristot\'eliciennes%
\footnote{\label{ari1}Being contrary to Aristotle was viewed as a
serious offense by the catholic hierarchy.  Thus, part of jesuit
Biancani's book was censored on the grounds that ``The addition to
Father Biancani's book about bodies moving in water should not be
published since \emph{it is an attack on Aristotle and not an
explanation of him} (as the title indicates).  Neither the conclusion
nor the arguments to prove it are due to the author, but to Galileo.
And it is enough that they can be read in Galileo's writings.  It does
not seem to be either proper or useful for the books of our members to
contain the ideas of Galileo, especially when they are contrary to
Aristotle.''  (quoted in \cite[p.\;162]{De03})}
et les difficult\'es th\'eologiques qui apporteront bient\^ot des
arguments solides aux adversaires des nouvelles orientations''
\cite[p.\;313--314]{In02}.
\end{quote}
The theological difficulties involved
included the eucharist:
\begin{quote}
nous ne savons pas grand-chose de la d\'efense de l'atomisme tent\'ee
par Rossetti pour surmonter les difficult\'es que cette doctrine
soulevait pour une interpr\'etation traditionnelle du Myst\`ere de
l'Eucharistie.  (ibid.)
\end{quote}
Indivisibles similarly were perceived as a theological threat and
opposed on doctrinal grounds in the 17th century \cite{Fe03}.  The
opposition was spearheaded by clerics and more specifically by the
jesuits.  Tracts opposing indivisibles were composed by jesuits Paul
Guldin, Mario Bettini, and Andr\'e Tacquet \cite[p.\;291]{Re87}.
P.\;Mancosu writes:
\begin{quote}
Guldin is taking Cavalieri to be composing the continuum out of
indivisibles, a position rejected by the Aristotelian orthodoxy as
having atomistic implications. \ldots{} Against Cavalieri's
proposition that ``all the lines" and ``all the planes" are magnitudes
- they admit of ratios - Guldin argues that ``all the lines \ldots{}
of both figures are infinite; but an infinite has no proportion or
ratio to another infinite."  \cite[p.\;54]{Ma96}
\end{quote}
Tacquet for his part declared that the idea of quantity composed of
indivisibles makes war upon geometry to such an extent, that if it is
not to destroy it, it must itself be destroyed; see
\cite[p.\;205]{Fe92}, \cite[p.\;119]{Al14}.

\section{Social, religious, and military 17th century background}

In this section we seek to present the social, religious, and military
background of the first half of the 17th century Italy and France.
This background helps understand the professional and religious
constraints Fermat was operating under.

\subsection{Bans on indivisibles}

In 1632 (the year Galileo received summons to stand trial over
heliocentrism) the Society's Revisors General led by Jakob Bidermann
banned teaching indivisibles in their colleges \cite{Fe90},
\cite[p.\;198, 207]{Fe92}.  The proposition and the ban read as
follows:
\begin{quote}
Continuum permanens potest constare ex solis indivisibilibus physicis,
seu corpusculis atomis, habentibus partes mathematicas, in ipsis
designabiles, etiamsi realiter dicta corpuscula inter se
distinguantur.  Tempus quoque ex instantibus, \& qualitates intensae,
ex solis gradis indivisibilibus constant.

Hanc propositionem arbitramur, non modo repugnare communi Aristotelis
doctrina, sed etiam secundum se esse improbabilem, etc.
\cite[p.\;207]{Fe92}
\end{quote}
This can be translated as follows:
\begin{quote}
A permanent continuum may consist of physical indivisibles, or minimal
particles, alone, which have mathematical parts that can be designated
by themselves: even though in reality the said particles are distinct
from each other.  Time, too, [consists] of instants, and intense
qualities%
\footnote{This refers to a scholastic dispute as to whether qualities
that differ in intensity differ in number: is there a numerical
difference between the anger of a person who is very angry and one who
is only slightly so?}
consist of indivisible grades alone.

We consider that this proposition not only stands in opposition to the
common doctrine of Aristotle,%
\footnote{This passage of the ban alludes to the Aristotelian doctrine
of \emph{hylomorphism}; see note~\ref{ari1}.}
but is also improbable in itself, etc.
\end{quote}
Referring to this ban, Feingold notes:
\begin{quote}
Six months later, General Vitelleschi formulated his strong opposition
to mathematical atomism in a letter he dispatched to Ignace Cappon in
Dole: ``As regards the opinion on \emph{quantity made up of
indivisibles}, I have already written to the Provinces many times that
it is in no way approved by me and up to now I have allowed nobody to
propose it or defend it.''  \cite[pp.\;28--29]{Fe03} (emphasis added)
\end{quote}
To dispel any remaining doubts, Vitelleschi put his foot down:
\begin{quote}
``If it has ever been explained or defended, it was done without my
knowledge. Rather, I demonstrated clearly to Cardinal Giovanni de Lugo
himself that I did not wish our members to treat or disseminate that
opinion.'' (ibid.)
\end{quote}
Indivisibles were placed on the Society's list of \emph{permanently}
banned doctrines in 1651 \cite{He96}.  One might wonder why a society
dedicated to furthering the pope's cause on the battlefield of ideas
would be interested in mathematics in the first place, whether
Euclidean or otherwise.  As Amir Alexander colorfully puts it,
``Ignatius of Loyola, founding father of the Society \ldots was not
enamored of mathematics'' \cite[p.\;52]{Al14}.  The answer, as argued
by Alexander, is Clavius.  It was C.\;Clavius (1538--1612) who
engineered the acceptance of the new calendar by pope Gregory\;13 and
the catholic clergy.  The development of the calendar patently
required a great deal of mathematical competence.  Riding the wave of
that success, Clavius was able to convince the jesuit hierarchy to
incorporate Euclid into their basic curriculum.

\subsection{``Vide int\'egral''}

The holy war led by jesuits Bettini, Bidermann, Grassi, Guldin,
Inchofer, Pallavicino, Tacquet, Vitelleschi, and others to purge Italy
of atoms and indivisibles bore fruit, and by 1700 Italy was a
mathematical desert:
\begin{quote}
en 1700, c'est le vide int\'egral en ce qui concerne la pratique des
math\'ematiques nouvelles en Italie\ldots{} \cite[p.\;183]{Ro91}
\end{quote}
It took a sustained multi-year campaign orchestrated by Leibniz and
mobilizing everybody from Varignon in Paris, Johann Bernoulli in
Basel, the Hanover ambassador Bothmer, to Army Chief Schulenburg in
order to install a leibnizian \emph{oltramontani}, Jakob Hermann and
subsequently Nicolas Bernoulli, in the mathematics chair at Padua,
eventually sparking a mathematical renaissance in Italy:
\begin{quote}
L'en\-seignement d'Hermann \`a la chaire de math\'ematiques de Padoue
est reconstituable gr\^ace aux apports de plu\-sieurs fonds \ldots{}
Les annonces des Rotuli de l'universit\'e concernent les intitul\'es
des cours annuels, qui changent totalement de ce \`a quoi on \'etait
habitu\'e et qui ne s'en tiennent plus aux commentaires traditionnels
de la g\'eo\-m\'etrie d'Euclide.  \cite[p.\;186]{Ro91}
\end{quote}
The jesuits, unfazed, stuck to their Euclidean guns.  The effects of
anti-indivisible bans were still felt in the 18th century, when most
jesuit mathematicians adhered to the methods of Euclidean geometry, to
the exclusion of the new infinitesimal methods:
\begin{quote}
\ldots le grand nombre des math\'ematiciens de [l'Ordre] resta
jusqu'\`a la fin du XVIII$^e$ si\`ecle profond\'ement attach\'e aux
m\'ethodes euclidiennes.  \cite[p.\;77]{Bo27}
\end{quote}

\subsection{Grassi and Guldin}

A critique of Galileo's indivisibles penned by jesuit Orazio Grassi is
described by Redondi as follows:
\begin{quote}
As for light - composed according to Galileo of indivisible atoms,
more mathematical than physical - in this case, logical contradictions
arise.  Such indivisible atoms must be finite or infinite.  If they
are finite, mathematical difficulties would arise.  If they are
infinite, one runs into all the paradoxes of the separation to
infinity which had already caused Aristotle%
\footnote{\label{ari2} See note~\ref{ari1} on the role of
aristotelianism.}
to discard the atomist theory\ldots{} \cite[p.\;196]{Re87}
\end{quote}
This criticism appeared in the first edition of Grassi's book
\emph{Ratio ponderum librae et simbellae}, published in Paris in 1626.
According to Redondi, this criticism of Grassi's
\begin{quote}
exhumed a discounted argument, copied word-for-word from almost any
scholastic philosophy textbook\ldots The Jesuit mathematician [Paul]
Guldin, great opponent of the geometry of indivisibles, and an
excellent Roman friend of [Orazio] Grassi, must have dissuaded him
from repeating such obvious objections.  Thus the second edition of
the \emph{Ratio}, the Neapolitan edition of 1627, omitted as
superfluous the whole section on indivisibles.  \cite[p.\;197]{Re87}
\end{quote}
Grassi and Guldin may have disagreed about the tactics of fighting
indivisibles, but they agreed about the need to do so (see also
Section~\ref{s63}).

\subsection{War and pope}
\label{s39}

There is an aspect of political history that is essential for
understanding the background for the events surrounding Galileo's
trial.  The protestant Swedish king Gustavus Adolphus was routing the
catholic Habsburgs on the battlefields of Germany until the situation
came to a head in 1632 when he seemed to threaten Rome itself.

The pope Urban 8 was faced with a situation close to a coup d'etat
vis-a-vis a group of cardinals led by the Spanish ambassador.  At this
point in time the pope was forced to reorient his policies away from
the French who had been allied with the Swedes, and toward the
Spanish, with the jesuits gaining greatly in influence.

At the same time Galileo's fortunes sank.  The ascendancy of the
jesuits was manifest by the 1633 trial of Galileo (see
Section~\ref{f9}), formerly a prot\'eg\'e of the pope, and the active
suppression of anything related to atomism or indivisibles.  

That this occurred in the first few years of Fermat's creative
mathematical activity may well have influenced his presentation of the
method of maxima and minima and the method of tangents, particularly
in view of sinister tensions with transubstantiation; see
Section~\ref{f13b}.  Such issue may have also influenced the
presentation of James Gregory's and degli Angeli's work, as we argue
in Sections~\ref{s4} and \ref{s36}.

\subsection{Scheiner vs Galileo}
\label{f9}

The trial of Galileo was provoked most likely by a denunciation from
jesuit C.\;Scheiner (see \cite{Fe91}) who could not conceal his glee
at the verdict \cite[p.\;109]{Fe91}.  Assuming Scheiner was the source
of the denunciation, Festa argues against Redondi's thesis that the
atomism/eucharist issue was the hidden agenda of the Galileo trial.
Festa offers two objections among others:
\begin{enumerate}
\item
\label{o1}
given that the G3 denunciation (see Section~\ref{s5}) dates from
around 1624, the Inquisition would not have delayed issuing a summons
until as late as 1632 if the real issue were G3 and atomism/eucharist;
and
\item
\label{o2}
if on the other hand the denunciation originated with Scheiner, the
hidden agenda hypothesis becomes implausible in view of Scheiner's
apparent admiration for Gassendi and his book based on atomist theses.
\end{enumerate}
Both of these objections, however, are answered in Redondi's book
itself.  As far objection~\eqref{o1} is concerned, the proceedings
against Galileo were not initiated earlier because the pope was
opposed to the idea, and the jesuits were too weak to impose them
until the dramatic turn of historical events in 1632 (see
Section~\ref{s39}).  As far as objection~\eqref{o2} is concerned, even
if Scheiner personally were not bothered by atomism, many others in
Rome were (including the author of the EE291; see below), and the
atomism/eucharist issue would have been a natural one to include in
the indictment.

As far as the present article is concerned, what is relevant is that
all agree that
\begin{quote}
les prises de position des j\'esuites contre l'atomisme au
XVII$^{\text{e}}$ si\`ecle n'avaient pas attir\'e jusqu'ici
l'attention de beaucoup d'historiens.  C'est avec la d\'ecouverte du
`G3' et la parution du livre de Redondi qu'un int\'er\^et pour ce
probl\`eme s'est manifest\'e \cite[p.\;116]{Fe91}.
\end{quote}
The subsequent discovery of the EE291 denunciation makes irrelevant
Festa's objection~\eqref{o1} and lends support to Redondi's
hypothesis; see Section~\ref{s63}.

\subsection{Degli Angeli in Italy}
\label{s4}

The jesuat mathematician degli Angeli defended the method of
indivisibles against the criticisms of jesuit scholars.  Less cautious
than Cavalieri, in his rebuttal of Tacquet's criticism degli Angeli
wrote:
\begin{quote}
``\ldots if in order to approve the method of indivisibles, the
composition of the continuum from indivisibles is necessarily
required, then certainly this doctrine [i.e., continuum made out of
indivisibles] is only strengthened in our eyes'' (degli Angeli as
translated in \cite[p.\;169--170]{Al14}; cf.\;\cite[p.\;205]{Fe92}).
\end{quote}

Both indivisibles and degli Angeli himself appear to have been
controversial at the time in the eyes of the jesuit order, which on
several occasions banned indivisibles from being taught in their
colleges.  We already mentioned in Section~\ref{s3} that in 1632 (the
year Galileo received summons to stand trial over heliocentrism) the
Society's Revisors General led by Jakob Bidermann banned teaching
indivisibles in their colleges \cite{Fe90}, \cite[p.\;198, 207]{Fe92}.
In 1651 indivisibles were placed on the Society's list of
\emph{permanently} banned doctrines \cite{He96}.

\subsection{Gregory's departure from Italy}
\label{s36}

It seems that James Gregory's 1668 departure from Padua was well
timed.  Indeed, his teacher degli Angeli's jesuat order (to which
Cavalieri had also belonged) was not treated with clemency but on the
contrary suppressed by papal brief of Clement\;9 in the same year.
Alexander describes the suppression as ``stunningly violent''
\cite[p.\;171]{Al14}.  The suppression cut short degli Angeli's output
on indivisibles.  Gregory's own books were suppressed at Venice,
according to a letter from John Collins to Gregory dated 25 november
1669, in which he writes:
\begin{quote}
One Mr.\;Norris a Master's Mate recently come from Venice, saith it
was there reported that your bookes were suppressed, not a booke of
them to be had anywhere, but from Dr.\;Caddenhead to whom application
being made for one of them, he presently sent him one (though a
stranger) refusing any thing for it. \cite[p.\;74]{Tu39}
\end{quote}
In a 1670 letter to Collins, Gregory writes:
\begin{quote}
I shall be very willing ye writ to Dr Caddenhead in Padua, for some of
my books.  In the mean time, I desire you to present my service to
him, and to inquire of him if my books be suppressed, and the reason
thereof.  (Gregory to Collins, St Andrews, march 7, 1670, in Turnbull
p.\;88)
\end{quote}
This passage indicates that Gregory sought an explanation for the
suppression of his books.  We are not aware of any clarification he
may have received on this issue.  In a letter to Gregory, written in
London on 29~september 1670, Collins wrote:
\begin{quote}
Father Bertet%
\footnote{\label{f13}Jean Bertet (1622-1692), jesuit, quit the Order
in\;1681.  In 1689 Bertet conspired with Leibniz and Antonio
Baldigiani in Rome to have the ban on Copernicanism
lifted. \cite{Wa12}}
sayth your Bookes are in great esteeme, but not to be procured in
Italy. (Turnbull p.\;107)
\end{quote}
The publishers' apparent reluctance to get involved with Gregory's
books may also explain degli Angeli's silence on indivisibles
following the suppression of his order, but it is hard to say anything
definite in the matter until the archives at the Vatican dealing with
the suppression of the jesuat order are opened to independent
researchers.  Certainly one can understand Gregory's own caution in
the matter of exploiting the infinitely small.  Nonetheless, Gregory
exploited some identifiably infinitary procedures in his work; see
\cite{17b}.

Beyond the borders of a catholic Italy, John Wallis introduced the
symbol~$\infty$ for an infinite number in his book \emph{Arithmetica
Infinitorum} \cite{Wa56}, and exploited an infinitesimal number of the
form~$\frac{1}{\infty}$ in area calculations \cite[p.\;18]{Sc81}, over
a decade before the publication of Gregory's \emph{Vera Circuli}.  At
about the same time, Barrow ``dared to explore the logical
underpinnings of infinitesimals:''
\begin{quote}
Barrow, who dared to explore the logical underpinnings of
infinitesimals, was certainly modern and innovative when he publicly
defended the new mathematical methods against Tacquet and other
mathematical ``classicists'' reluctant to abandon the Aristotelian
continuum.  And after all, to use historical hindsight, it was the
non-Archimedean structure of the continuum linked to the notion of
infinitesimal and advocated by Barrow that was to prove immensely
fruitful as the basis for the Leibnizian differential
calculus. \cite[p.\;244]{Ma89}.
\end{quote}
On the other side of the Channel, Wallis and Barrow were freer than
degli Angeli and Fermat to pursue the infinitely small.

The international nature of the jesuit organisation could not but
affect what can or cannot be said in catholic France.  Fermat's friend
Lalouv\`ere at Toulouse was both a jesuit and a mathematician opposed
to indivisibles.  Lalouv\`ere actively interacted with Fermat in
matters mathematical and would surely serve as a reminder should
Fermat think of engaging in activities (such as any explicit
endorsement of Cavalieri's doctrines) the jesuits considered
objectionable.

\subsection{Council of Trent, Galileo, and G3}
\label{s5}

An anonymous denunciation of Galileo, labeled G3, dating from the
1620s and preserved in the Vatican archives specifically connects
Galileo's atomism to an ``error condemned by'' the Council of Trent
Session 13, canon 2.  The G3 document stated:
\begin{quote}
\ldots one will also have to say according to this doctrine that there
are the very tiny particles with which the substance of the bread
first moved our senses, which if they were substantial (as Anaxagoras
said, and this author [i.e., Galileo] seems to allow on page 200, line
28), it follows that in the Sacrament there are substantial parts of
bread or wine, which is the error condemned by the Sacred Tridentine
Council, Session 13, Canon 2.  (G3 cited in \cite[p.\;334]{Re87})
\end{quote}
Having earlier mentioned the atomic theories of Anaxagoras and
Democritus (ancient authorities espousing atomism), the author of G3
goes on to impute to Galileo an alignment with Democritus:
\begin{quote}
Or actually, if they were only sizes, shapes, numbers, etc., as he
[i.e., Galileo] also seems clearly to admit, agreeing with Democritus,
it follows that all these are ac\-cidental modes, or, as others say,
shapes of quantity.  (ibid.)
\end{quote}
This passage indicates that Galileo's science, allegedly following the
atomism of Democritus, was viewed as a threat to canon 13.2 related to
one of the disagreements between Rome and the protestants (see
Section~\ref{s63}), and was therefore controversial.

Galileo held that sensible qualities (color, taste, etc.) existed in
minds, but not in bodies; they were, in Galileo's phrase, ``nothing
but names so far as the object in which we place them is concerned,''
with no necessary connection to properties in bodies.  Meanwhile
catholic doctrine held that a miracle was necessary to preserve the
color, taste, etc. of the host after it had been miraculously changed
to the body of the nazarene.  Grassi pointed out the incompatibility
of the two doctrines.

\subsection{Document EE291 on Galileo's atomism}
\label{s10}

In 1999 Mariano Artigas discovered an additional anonymous
denunciation of Galileo.  The new document is labeled EE291 and dates
from 1632.  The issue is again incompatibility with the eucharist.
The author of EE291 is thought to be the jesuit Melchior Inchofer.
The criticisms voiced in G3 were echoed in Inchofer's 1632
denunciation EE291 in the following terms:
\begin{quote}
2. [Galileo] errs when he says that it is not possible to conceptually
separate corporeal substances from the accidental properties that
modify them, such as quantity and those that follow quantity.  Such an
opinion is absolutely contrary to faith, for instance in the case of
the eucharist, where quantity is not only really distinguished from
substance but, moreover, exists separately. 

\medskip\noindent 3. He errs when he says that taste, smell, and
colour are pure names, etc.  \newline (as cited in \cite{AMS})
\end{quote}
The denunciation was filed in 1632 only a few months before Galileo
received summons to stand trial.

\subsection{Concerning Canon 13.2}
\label{s63}

In 1551 the said canon stipulated:
\begin{quote}
CANON II.  If any one saith, that, in the sacred and holy sacrament of
the eucharist, the substance of the bread and wine remains conjointly%
\footnote{``panis et vini una cum corpore et sanguine'' in the
original Latin.  This is the key term here - this is the doctrine
favored by Luther - \emph{consubstantiation}.  It was favored also by
Scotus and Ockham, who rejected it only because it was inconsistent
with the 1215 Lateran council.  Artigas et al.\;note that ``the
concept of substance [in\;13.2] was borrowed from Aristotelian
philosophy'' \cite{AMS}.  The substance/form dichotomy is the content
of the Aristotelian doctrine of hylomorphism.}
with the body and blood of [the nazarene], and denieth that wonderful
and singular conversion of the whole substance of the bread into the
Body, and of the whole substance of the wine into the Blood--the
species Only of the bread and wine remaining--which conversion indeed
the Catholic Church most aptly calls Transubstantiation; let him be
\emph{anathema}. (Council of Trent, Session 13) (emphasis added)
\end{quote}
The Council of Trent said nothing of mathematical indivisibles, but
the jesuits interpreted 13.2 as anathemizing both physical and
mathematical indivisibles, a distinction they may not have drawn.
Whether they did or not is irrelevant for our purposes.

It is less than apparent why this canon entails the condemnation of
indivisibles, so a word of explanation is in order.  The key to
understanding the conflict between Galilean atoms and 13.2 is the
adverb \emph{conjointly} (``una cum'' in the original), which alludes
to the doctrine of \emph{con}substantiation, one of the alternatives
to transubstantiation.  According to consubstantiation, the doctrine
condemned by Aquinas as heresy but favored by Luther, the substance of
the bread and the substance of the nazarene's body are present
together in the eucharist.  

In transubstantiation, by contrast, the substance of the bread is
changed by a miracle into the substance of the nazarene's body; see
\cite{Mc68} for details.  Now, Galileo's atomism holds that the
sensible qualities of the bread are no more than the effect of the
atoms of the bread acting upon the sense organs; i.e., they are not
qualities of the bread itself.  Since it is inconceivable that the
faithful should be deceived through sensations that correspond to
nothing real in the eucharist, the atomist must insist that the atoms
of bread continue to exist consubstantially, i.e.\;together with the
body of the nazarene, to cause those sensations; see \cite{Ch02} for
details.

Galileo's critic, Grassi, scoffed that for the preservation of mere
\emph{names} no miracle would be required.%
\footnote{Grassi's objections (see \cite[pp.\;100--101]{Fe91}) closely
parallel those contained in G3, leading Redondi to conjecture that
Grassi was in fact the author of\;G3.  Many Galileo scholars, however,
feel that there is insufficient evidence for this, and specifically
reject Redondi's claim of similarity of handwriting.  See
Section~\ref{f9} for further details.}
A condemnation of atomism as subversive of the eucharist was issued by
jesuit Sforza Pallavicino \cite[p.\;115]{Fe91}; \cite[p.\;203]{Fe92}.
A further denunciation was issued in 1632; see Section~\ref{s10}.

\subsection{Old and new heresy}

Incompatibility with canonical doctrine was not merely a theoretical
issue.  Thus, David Derodon's book was burned as the result of a
denunciation and pressure exerted by the jesuits \cite[p.\;274]{Re87}.
Derodon's 1655 book, \emph{Dispute sur l'Eucha\-ristie}, affirmed the
spiritual presence over the real presence of the nazarene (ibid.).
Redondi refers to it as the \emph{old heresy}; this was first codified
in 1079 by the synod of Rome.

Another illustration of the inherent dangers of publication may have
been familiar to Fermat through his friendship with the d'Espagnets.
Jean d'Espagnet collected scientific manuscripts and eventually passed
on his library to his son Etienne.  Jean wrote a book on alchemy under
a pseudonym because in the atmosphere of the counterreformation it may
have been dangerous to publish it under his own name
\cite[p.\;xxii]{Wi99}.  Fermat first became acquainted with the work
of Vieta through this private library in Bordeaux in the late 1620s.
Many of Fermat's manuscripts were deposited with Etienne d'Espagnet
for safekeeping and eventually reached Mersenne.  

Fermat may well have come to the conclusion that publishing books was
too risky, and chose a safer method of disseminating his discoveries
through private letters and manuscripts.  The fate of alchemists
imprudent enough to publish under their real names in the 1620s is
analyzed in \cite{Ka02}.

Atomism and indivisibles were seen by some in the catholic hierarchy,
including Grassi, the author of G3, Inchofer, and others, as a threat
to the said canon~13.2.

\subsection{Reception of Galileo trial in France}
\label{s411}

Lewis's monograph \cite{Le06} analyzes the reception of the Galileo
trial in France, while the monograph \cite{Pa04} presents a broader
European perspective.  The salient points for us concern Fermat's
correspondents in Paris, Carcavy and Mersenne.  Marin Mersenne
contributed in a major way to the diffusion of Galileo's ideas in
France \cite[Chapter\;5]{Le06}; see also \cite[p.\;36]{Pa04}.  Carcavy
sought to help Galileo disseminate his works \cite[Chapter\;4]{Le06};
see also \cite[p.\;38]{Pa04}.

Fermat's letter to Mersenne indicates that Fermat possessed a copy of
Galileo's \emph{Due Nuove Scienze} as early as 10 august 1638
\cite[p.\;137]{Le06}.  Fermat was involved in discussions of Galileo's
science of motion, as we know from a letter from Pardies to Oldenburg:
\begin{quote}
``Father La Loub\`ere came in on this and demonstrated that motion
could perfectly well occur in the hypothesis of Father Cazr\'e
provided that the body did not pass through all degrees of
slowness. In fact he claims that as the weight of a body is determined
at a certain degree of force, this weight also pushes the body
downwards with a certain degree of speed from the beginning of its
fall; and this seemed so reasonable that Mr. Fermat himself found no
fault with it.''  (quoted in \cite[p.\;217]{Pa03})
\end{quote}
Here Father La Loub\`ere is none other than Fermat's friend
Lalouv\`ere, enemy of indivisibles \cite[p.\;267]{De15}, who appears
to have discussed a variety of scientific subjects with Fermat.

Yet another Galileo--Fermat connection is analyzed in \cite{Ro11}.

\section{Fermat at Castres and conclusion}
\label{s6}

Fermat was personally in charge of interdenominational issues as part
of his professional responsibilities at the Parliament of Toulouse:
\begin{quote}
[Fermat] fait partie de la ``Chambre de l'\'Edit'', pr\'evue par
l'\'edit de Nantes pour juger paritairement des litiges entre
catholiques et protestants.  Il fut un homme de compr\'ehension et de
conciliation, au lendemain des guerres de religion.  \cite{Fe02}
\end{quote}
Fermat acted as a representative of the Parliament of Toulouse to the
\emph{Chambre de l'\'Edit} in Castres in charge of conflicts involving
Protestants and Catholics.


Fermat's official position as mediator may help explain his reticence
to comment on issues, like the nature of his~$E$ exploited in the
method of adequality, that could have involved him in controversy over
the canons; see Section~\ref{s5}.

\subsection{Jean de l'Hospital, \emph{(ancient) Conseiller}}
\label{s41}

Jean de l'Hospital was \emph{Conseiller} at the Parliament of Toulouse
until 1562, when he was expelled ``en raison de ses choix religieux''
\cite[pp.\;27--28]{Mo17}.  Following the 1572 Bartholomew massacre
\cite[p.\;42]{Mo17}, he became president of a judicial body
(\emph{chambre mi-partie}) that would move to Castres and eventually
be transformed into the \emph{Chambre de l'\'Edit}.  According to
historians who identify Claire de Long as Pierre de Fermat's mother,%
\footnote{See note~\ref{f8}.}
Jean de l'Hospital was his great-grandfather; see further in
Section~\ref{s44}.

\subsection{Fermat's poem}
\label{f13b}

Fermat kept a safe distance away from the subject of
con/transubstantiation (see Section~\ref{s63}) not only in his
mathematics but also in his poetry.  Fermat's poem was presented at
the Academy of Castres (of which Fermat was not officially a member)
in 1656 \cite[p.\;60]{Mo17}.

Fermat was inspired by Balzac's lucid writings on some arcane points
of religious doctrine to attempt to do the same with regard to the
events surrounding the execution of the nazarene.  The resulting poem
was meant to enlighten and inspire the faithful with regard to this
central event of both catholic and protestant traditions.  It would
therefore have been natural to mention the eucharist as well in the
poem.  The fact that Fermat does not do so directly is significant.

Fermat's 101-line hexameter poem concerning the events leading up to
the execution carefully steers away from the subject of the eucharist,
which is however considered by Fermat's coreligionists to be central
to the events in question.  The execution is evoked at line 36,
followed at lines 55--56 by a flashback to his companions falling
asleep (after the supper when he reportedly uttered the words
\emph{hoc est corpus, etc.}), but there is no mention in the poem of
\emph{hoc est corpus}.  

The related interdenominational tensions are well illustrated by the
following passage penned by anglican archbishop John Tillotson in
1694: ``In all probability those common juggling words of \emph{hocus
pocus} are nothing else but a corruption of \emph{hoc est corpus}, by
way of ridiculous imitation of the [catholic] priests \ldots{} in
their trick of Transubstantiation.''%
\footnote{See also
\url{http://www.etymonline.com/index.php?term=hocus-pocus}}

\subsection{Academy of Castres}

Further details on the Academy of Castres are provided by Spiesser:
\begin{quote}
\ldots c'est surtout \`a Castres que se retrouvent nombre de
correspondants et d'interlocuteurs de Fermat. Peut-\^etre est-ce une
des raisons pour lesquelles le magistrat affectionnait tout
particuli\`erement d'y s\'ejourner. Ce bastion de la religion
r\'eform\'ee en albigeois, au d\'ebut du XVIIe si\`ecle, accueille le
Chambre de l'Edit entre 1632 et 1670; et ce sont les membres de cette
Chambre, parmi lesquelles se trouvent beaucoup de huguenots, qui ont
constitu\'e les forces vives de l'Acad\'emie fond\'ee en 1648 par Paul
Pellisson, et dont les ann\'ees florissantes se situent dans cette
p\'eriode.  \cite[p.\;292]{Sp16}
\end{quote}
Spiesser continues with a more detailed description of Fermat's circle
of acquaintances at Castres:
\begin{quote}
Le r\'eseau castrais de Pierre de Fermat s'organise autour de
l'Acad\'emie locale.  Bien qu'\`a notre connaissance, il n'ait jamais
\'et\'e membre de cette institution, elle constitue la pierre de
touche de ses alliances amicales et litt\'eraires. L'avocat Pierre
Saporta, membre de cette petite acad\'emie et traducteur du `Trait\'e
sur la mesure des eaux courantes' du p\`ere b\'en\'edictain Benedetto
Castelli, ainsi que d'un ouvrage de Torricelli sur le `mouvement des
eaux', admire Fermat.  (ibid.)
\end{quote}

Given that Fermat's main professional occupation as \emph{Conseiller}
at the Parliament of Toulouse and its \emph{Chambre de l'\'edit} at
Castres involved delicate matters of interdenominational dialogue,
Fermat could hardly have allowed himself to get involved in anything
related to issues of controversy dividing the different and sometimes
warring (see Section~\ref{s39}) communities, such as the issues of
transubstantiation and its potential tensions with atomism (and, by
implication, with indivisibles and the infinitely small), seen at the
time as potentially contrary to canon established at the Council of
Trent in the previous century, with dire consequences for failing to
toe the canonical line.

\subsection
{The lesson of the Marranos of Toulouse}
\label{s44}

A case in point (of failing to toe the canonical line) is the tale of
18 Toulouse Marranos, judged by the honorable judge de Fermat
\emph{fils} among others and sentenced in 1685.  The Marranos escaped
in time and only their effigies were burned \emph{in absentia} in
1692: ``Les accus\'es ont eu raison de s'enfuir.  Toulouse les aurait
certainement envoy\'es au b\^ucher.''%
\footnote{``The accused were right to flee.  Toulouse would have
certainly sent them to the stake.''  One of the 18 Marranos, named
Roque de Leon, was an ancestor of Jacques Blamont; see
\cite[p.\;17]{Bl00}.}
\cite[p.\;352]{Bl00}.

When Pierre Fermat was a teenager, Lucilio Vanini (1585--1619) was
executed at Toulouse on charges of blasphemy \cite[p.\;62]{Mo17}.

No \emph{Conseiller} other than Pierre de Fermat of the Parliament of
Toulouse was interested in mathematics and no member in the protestant
academy of Castres was a mathematician, but a rumor of dubious
dealings in consubstantial natural philosophy may well have ruined
Fermat, who had his share of enemies at the Parliament, including the
President Fieubet.%
\footnote{We find the following period detail: ``\ldots l'intendant du
Languedoc, Claude Bazin de Bezons, a \'emis en 1663 un jugement aussi
lapidaire que p\'ejoratif sur son [Fermat's] activit\'e de magistrat.
Dans une note adress\'ee au ministre Colbert sur les membres du
parlement de Toulouse, il d\'eclare que Fermat est `homme de beaucoup
d'\'erudition, a commerce de tous cost\'es avec des s\c cavans, mais
assez int\'eress\'e, n'est pas trop bon rapporteur et est confus'{}''
\cite[p.\;52]{Mo17}.  Bazin added that Fermat ``n'est pas des amys du
premier pr\'esident'' (ibid.), namely Gaspard de Fieubet.  President
Fieubet overruled Fermat in the matter of the priest Raymond Delpoy
and had Delpoy hanged.  ``Fermat, qui n'\'etait pas convaincu de sa
culpabilit\'e, en fut choqu\'e'' (ibid.).}
To put in bluntly, the honorable judge Pierre de Fermat risked ending
up on the wrong side of the bench (or suffering the fate of his
great-grandfather; see Section~\ref{s41}) had he pursued those
exceedingly little devils with excessive zeal.

The infinitely small were seen as inseparable from atomism, which was
seen as heretical because it contradicted canon 13.2 of the Council of
Trent, which was one of the major issues dividing catholics and
protestants, so much so that books and people were burned over this.

In a more liberal milieu, mathematicians like Wallis, Barrow, and
Huygens felt freer to speak of, and indeed freely spoke of, the
infinitely small, resuming the discussion where the once-flourishing
Italian school had left it.

\section*{Acknowledgments}

We are grateful to Catherine Goldstein, Israel Kleiner, Eberhard
Knobloch, David Schaps, and Maryvonne Spiesser for helpful comments.
We thank Thomas Willard for the information on Fludd and Kepler given
in note~\ref{f8b}.  M.\;Katz was partially supported by the Israel
Science Foundation grant no.\;1517/12.

\bigskip\noindent \textbf{Jacques Bair} is Professor Emeritus at the
University of Li\`ege (Belgium) who taught to students in economics
and in management.  He began as a specialist in convex geometry (see,
for example, Springer Lecture Notes in Mathematics, N 489 and 802 by
Bair and Fourneau).  His present interests are epistemology and the
teaching of mathematics.  His most recent work concerns the
foundations of mathematical analysis.

\bigskip\noindent \textbf{Mikhail G. Katz} (BA Harvard '80; PhD
Columbia '84) is Professor of Mathematics at Bar Ilan University,
Ramat Gan, Israel.  He observed a while ago that with Abraham
Robinson's passing, the quest for the ghosts of departed
\emph{quantifiers} in the work of Fermat, Leibniz, Euler, and Cauchy
began in earnest, as many a triumvirate scholar sought to transmogrify
their infinitesimal procedures into proto-Weierstrassian hocus-pocus.
He therefore resolved not to keep mum about it and with the help of
coauthors who were adequal to the task, launched a decidedly
nonstandard, long march through triumvirate historiography.  While in
the present article we examine the issue at stake with Fermat's
method, we take a fresh look at Leibniz in \cite{16a}, at Euler in
\cite{17a}, and at Cauchy in \cite{18b}.

\bigskip\noindent \textbf{David Sherry} is Professor of Philosophy at
Northern Arizona University, in the tall, cool pines of the Colorado
Plateau.  He has research interests in philosophy of mathematics,
especially applied mathematics and non-standard analysis.  Recent
publications include ``Fields and the Intelligibility of Contact
Action,'' \emph{Philosophy 90} (2015), 457--478.  ``Leibniz's
Infinitesimals: Their Fictionality, their Modern Implementations, and
their Foes from Berkeley to Russell and Beyond,'' with Mikhail Katz,
\emph{Erkenntnis 78} (2013), 571-625.  ``Infinitesimals, Imaginaries,
Ideals, and Fictions,'' with Mikhail Katz, \emph{Studia Leibnitiana
44} (2012), 166--192.  ``Thermoscopes, Thermometers, and the
Foundations of Measurement,'' \emph{Studies in History and Philosophy
of Science 24} (2011), 509--524.  ``Reason, Habit, and Applied
Mathematics,'' \emph{Hume Studies 35} (2009), 57-85.

\end{document}